\documentclass[a4paper,11pt]{article}

\usepackage{subfigure}
\usepackage{amsfonts}
\usepackage{amssymb}
\usepackage{amsthm}
\usepackage{amsmath}
\usepackage{a4wide}
\usepackage{mathrsfs}
\usepackage{epsfig}
\usepackage{tikz,fp,ifthen}
\usepackage{float}
\usepackage[all]{xy}
\usepackage[pagewise]{lineno}

\newcommand{\pdfgraphics}{\ifpdf\DeclareGraphicsExtensions{.pdf,.jpg}\else\fi}
\usepackage{graphicx}

\usepackage{color}
\definecolor{hanblue}{rgb}{0.27, 0.42, 0.81}
\definecolor{red}{rgb}{1.0, 0.0, 0.0}
\usepackage[colorlinks, citecolor=hanblue,linkcolor=red]{hyperref}

\usepackage{mathtools}
\usepackage{hyperref}

\renewcommand{\div}{{\rm div}\,}

\newcommand{\R}{\mathbb{R}}
\newcommand{\Z}{\mathbb{Z}}

\newcommand{\Ha}{\mathcal{H}}

\definecolor{green(munsell)}{rgb}{0.0, 0.66, 0.47}
\definecolor{inchworm}{rgb}{0.7, 0.93, 0.36}
\definecolor{icterine}{rgb}{0.99, 0.97, 0.37}
\definecolor{magicmint}{rgb}{0.67, 0.94, 0.82}
\definecolor{tigerseye}{rgb}{0.88, 0.55, 0.24}
\definecolor{lavenderblue}{rgb}{0.8, 0.8, 1.0}
\definecolor{ceruleanblue}{rgb}{0.16, 0.32, 0.75}
\definecolor{airforceblue}{rgb}{0.36, 0.54, 0.66}
\definecolor{tealgreen}{rgb}{0.0, 0.51, 0.5}
\definecolor{cream}{rgb}{1.0, 0.99, 0.82}
\definecolor{jonquil}{rgb}{0.98, 0.85, 0.37}
\usetikzlibrary{patterns}

\newcommand\restr[2]{{
  \left.\kern-\nulldelimiterspace 
  #1 
  \vphantom{\big|} 
  \right|_{#2} 
  }}

\numberwithin{equation}{section}

\theoremstyle{plain}

\newtheorem{thm}{Theorem}[section]

\newtheorem{lemma}[thm]{Lemma}
\newtheorem{prop}[thm]{Proposition}

\theoremstyle{definition}
\newtheorem{dfnz}[thm]{Definition}

\theoremstyle{remark}
\newtheorem{rem}[thm]{Remark}

\newtheorem{ex}[thm]{Example}

\definecolor{green(munsell)}{rgb}{0.0, 0.66, 0.47}
\definecolor{inchworm}{rgb}{0.7, 0.93, 0.36}
\definecolor{icterine}{rgb}{0.99, 0.97, 0.37}
\definecolor{magicmint}{rgb}{0.67, 0.94, 0.82}

\begin{document}
\pdfgraphics 

\pdfgraphics 

\title{On different notions of calibrations for minimal partitions and minimal networks in $\mathbb{R}^2$}

\author{
Marcello Carioni \footnote{Institut f\"ur Mathematik,
Universit\"at Graz, Heinrichstraße 36,
8010 Graz,
Austria}
\and 
Alessandra Pluda \footnote{Dipartimento di Matematica, Universit\`{a} di Pisa,
Largo Bruno Pontecorvo 5, 56127, Pisa, Italy}}

\date{}

\maketitle

\begin{abstract}
\noindent
Calibrations are a possible tool to validate the minimality of a certain candidate. 
They have been introduced in the context 
of minimal surfaces~\cite{berger, realcochains, CGE}
and adapted to the case of the Steiner problem in several variants. 
Our goal is to compare the different notions of calibrations
for the Steiner Problem and for planar minimal partitions appearing in~\cite{calicipi,lawmor,annalisaandrea}.
The paper is then complemented with remarks on the convexification of the problem, on
non--existence of calibrations and on calibrations in families.
\end{abstract}

\section{Introduction}

Let $S$ be a
collection of $n$ points $p_1,\ldots,p_n$ in the Euclidean plane.
We want to
find a connected set that contains $S$ whose length is minimal, namely
\begin{equation}\label{ste}
\inf  \{\Ha^1(K) : K \subset \R^2, \mbox{ connected and such that } S
\subset K\}\,.
\end{equation}
This latter is commonly known as the \emph{Steiner problem}.

Although the existence of minimizers is known,
finding explicitly a solution is
extremely challenging even numerically.
For this reason every method to determine solutions is welcome.
A classical tool is the notion of calibration,
introduced in the framework of minimal surfaces~\cite{berger, realcochains, CGE} 
(see also~\cite[\S~6.5]{morganbook} for an overview of the history of calibrations):
given $M$ a $k$--dimensional oriented manifold in $\R^{d}$, a calibration
for $M$
is a closed $k$--form $\omega$
such that $\vert\omega\vert\leq 1$ and $\langle \omega, \xi\rangle = 1$
for every $\xi$ in the
tangent space of $M$. 
The existence of a calibration for $M$ implies that the
manifold is area minimizing in its homology class. Indeed
given an oriented $k$--dimensional manifold $N$ such that $\partial M =
\partial N$
we have
\begin{equation*}
\mbox{Vol}(M) = \int_{M} \omega = \int_N \omega \leq \mbox{Vol}(N)\, ,
\end{equation*}
where we applied the properties required on the calibration $\omega$
and we used Stokes' theorem in the second equality.

\medskip

This definition of calibration is not suitable for the Steiner
Problem~\eqref{ste}
simply for the reason that
neither the competitors nor the minimizers of the problem 
admit an orientation which is compatible with their boundary.
To overcome this issue
several variants have been
defined
starting from the \emph{paired calibrations}
by Morgan and Lawlor in~\cite{lawmor},
where the Steiner problem is seen as a problem of minimal partitions.
In~\cite{annalisaandrea}
Marchese and Massaccesi
rephrase the Steiner Problem as a mass minimization for
 $1$--rectifiable currents with coefficients in a group
 and this leads to a suitable definition of calibrations
(see also~\cite{orlandi}).
Finally reviving the approach via covering space by Brakke~\cite{brakke}
(see~\cite{cover} for the existence theory)
another notion of calibrations
has been produced~\cite{calicipi}.

\medskip

A natural question is whether
the previously mentioned notions of calibrations are equivalent. In the first part of the paper we give an answer to it.
When the points of $S$ lie on the boundary of a convex set
(actually the only case in which paired calibrations are defined)
calibrations on coverings are nothing but paired calibrations.
On the other hand an equivalence does not exist between
calibrations on coverings and calibrations for currents with coefficients in $\mathbb{R}^n$;
in particular the notion of calibrations for currents
is stronger than the one on coverings.
In other words it is easier to find a calibration on coverings.

Let us now discuss in more depth the relation between the two notions.
The definition of calibrations for currents with coefficients in
$\mathbb{R}^n$
(see Definition~\ref{calicurrents})
depends on choice of the norm of $\mathbb{R}^n$ (see \cite{morgancluster} where different norms are used to study clusters with multiplicities).
The norm considered in~\cite{annalisaandrea} (see also \cite{morgancluster}), here denoted by $\Vert\cdot\Vert_\flat$,
is the one that produces the weakest notion of calibrations and still gives the equivalence with the Steiner problem
in $\mathbb{R}^d$ with $d\geq 2$: the ``best possible" norm in a certain
sense.
It turns out that this notion of calibration is stronger than the one
on coverings.
Indeed in Theorem~\ref{dacalicurrentacalicovering} we are able to prove that
if a calibration for a mass minimizing current with coefficients in $\R^n$ exists, then
there exists also a calibration for a perimeter minimizing set in a given
covering,
but the converse does not hold.
To prove a sort of converse one has to abandon the idea
of working in the general setting of Marchese and
Massaccesi~\cite{annalisaandrea}
and take full advantage of restricting to  $\mathbb{R}^2$.
To this aim we slightly change the mass minimization problem and
we define a different norm on $\mathbb{R}^n$ denoted by
$\Vert\cdot\Vert_\natural$
(the unit ball of $\Vert\cdot\Vert_\natural$ is smaller than the one of
$\Vert\cdot\Vert_\flat$ as one can see (at least in $\mathbb{R}^3$)
from their Frank diagram depicted in Figure~\ref{norm}).
The $\Vert\cdot\Vert_\natural$ notion of calibration is equivalent with the definition of calibrations on coverings in $\R^2$.

\medskip

The second part of the paper has a different focus and it can be seen as a
completion of~\cite{calicipi} as we restrict our attention to
calibrations on coverings.
In Theorem~\ref{impli} we prove that the existence of a calibration for a
constrained set $E$
in a covering $Y$ implies the minimality of $E$ not only among
(constrained) finite perimeter sets,
but also in the larger class of finite linear combinations of characteristic
functions of finite perimeter sets (satisfying a suitable constraint).
This apparently harmless result has some remarkable consequences.

First of all it is directly related to the convexification of the problem
naturally associated with the notion of calibration.
This convexification $G$ is the
so--called  ``local convex envelope" and it has been defined
by Chambolle, Cremers and Pock.
In~\cite{chambollecremerspock}  they are able to prove that it is
the tightest among the convexifications with an integral form.
Unfortunately it does not coincide with the convex envelope of the
functional,
whose characterization is unknown.
We show that $G$ equals the total variation
on constrained $BV$ functions with a finite number of values.
In other words, the local convex envelope ``outperforms" the total variation
only when evaluated on constrained $BV$ functions
whose derivatives have absolutely continuous parts 
with respect to $\mathscr{L}^2$.

As a second consequence of Theorem~\ref{impli}
we produce a counterexample to the existence of calibrations.
It has already been exhibited in the setting of normal currents by
Bonafini~\cite{bonafini}
and because of the result of Section~\ref{equivalence} we had to
``translate" it in our framework.
It is specific to the case
in which $S$ is composed of five points, the vertices of a regular pentagon, and
cannot be easily generalized to vertices of  other regular polygons.

\medskip

We summarize here the structure of the paper.
In Section~\ref{problem} we recap the different approaches to the
Steiner Problem and the consequent notions of calibrations.
Section~\ref{equivalence}
is devoted to the relations among different definitions of calibrations.
Then in Section~\ref{convexification} we generalize the theorem
``existence of calibrations implies minimality", and this allows us to
complement a result by Chambolle, Cremers and Pock on the
convexification of the problem.
An example of nonexistence of calibrations is given in
Section~\ref{nonexistence}.
The paper is concluded with some remarks about the calibrations
in families presented in~\cite{calicipi} that underline the effectiveness
of our method.

\section{Notions of calibrations for minimal Steiner networks}\label{problem}

In this section we briefly review the approaches to the Steiner Problem and the related notions of calibrations presented in the literature~\cite{calicipi, lawmor, annalisaandrea}.
 
\subsection{Covering space approach \cite{cover,brakke,calicipi}}

We begin by explaining the approach via covering space by Brakke~\cite{brakke}
and Amato, Bellettini and Paolini~\cite{cover}.
They proved that minimizing the  perimeter  among 
constrained sets on a suitable defined 
covering space of $\mathbb{R}^2\setminus S=:M$ 
is equivalent minimizing the length among all networks that connect the point of $S$.
We refer to both~\cite{cover} and~\cite{calicipi} for details.
\medskip

Consider a \emph{covering space} $(Y, p)$ where $p:Y\to M$ 
is the projection onto the base space.
Consider  $\ell$  a loop in $\mathbb{R}^2$  around at most $n-1$ points of $S$.
Heuristically $Y$ is composed of $n$ copies of $\mathbb{R}^2$ 
(the sheets of the covering space) glued in 
such a way that going along $p^{-1}(\ell)$ in $Y$,
one ”visits” all the $n$ sheets.
We avoid repeating here the explicit construction of $Y$ presented in~\cite{cover}
but it is relevant to keep in mind 
how points of different copies of $\mathbb{R}^2\setminus S$ are identified.
First the $n$ points of $S$ in $\R^2$ are connected 
with a \emph{cut} $\Sigma \subset \Omega$ 
given by the union of injective Lipschitz curves $\Sigma_i$ from $p_i$ to $p_{i+1}$
(with $i\in\{1,\ldots, n-1\}$) not intersecting each other.
Then $\Sigma_i$ is lifted to all the $n$ sheets of $M$ and 
the points of $\Sigma_i$ of the $j$--th sheet are identified with points of
$\Sigma_i$ of the $k$--th sheet 
via the equivalence relation 
\begin{equation*}
k\equiv j+i\, (\mathrm{mod} \;n) \qquad\text{with}\; i=1,\ldots, n-1\; \text{and}\; j=1,\ldots, n\,.
\end{equation*}
This equivalence relation produces a non--trivial  covering of $M$.

\begin{figure}[h]
\begin{center}
\begin{tikzpicture}[scale=0.8]
\draw[white]
(-2,-1.85)to[out= 0,in=180, looseness=1] (2,-1.85)
(-2,1.85)to[out= 0,in=180, looseness=1] (2,1.85)
(-2,-1.85)to[out= 90,in=-90, looseness=1] (-2,1.85)
(2,-1.85)to[out= 90,in=-90, looseness=1] (2,1.85);
\fill[black](0,1) circle (1.7pt);
\fill[black](-0.85,-0.5) circle (1.7pt);
\fill[black](0.85,-0.5) circle (1.7pt);
\path[font=\normalsize]
(-0.86,-0.5)node[left]{$p_1$}
(0.86,-0.5)node[right]{$p_2$}
(0,1)node[above]{$p_3$};
\path[font=\small]
(-1.5,-1.85) node[above]{$\mathbb{R}^2$};
\end{tikzpicture}\quad
\begin{tikzpicture}[scale=0.8]
\draw[black!50!white, dashed]
(-2,-1.85)to[out= 0,in=180, looseness=1] (2,-1.85)
(-2,1.85)to[out= 0,in=180, looseness=1] (2,1.85)
(-2,-1.85)to[out= 90,in=-90, looseness=1] (-2,1.85)
(2,-1.85)to[out= 90,in=-90, looseness=1] (2,1.85);
\draw[color=black, dashed, very thick]
(0.86,-0.5)to[out= 90,in=-30, looseness=1] (0,1)
(-0.86,-0.5)to[out= -30,in=-150, looseness=1] (0.86,-0.5);
\draw[color=black,scale=1,domain=-3.141: 3.141,
smooth,variable=\t,shift={(0,-0.7)},rotate=0]plot({0.4*sin(\t r)},
{0.4*cos(\t r)});
\fill[black](0,1) circle (1.7pt);
\fill[black](-0.85,-0.5) circle (1.7pt);
\fill[black](0.85,-0.5) circle (1.7pt);
\path[font=\normalsize]
(-0.86,-0.5)node[left]{$p_1$}
(0.86,-0.5)node[right]{$p_2$}
(0,1)node[above]{$p_3$};
\path[font=\small]
(-1.5,-1.85) node[above]{$D_1$};
\filldraw[fill=white, color=black, pattern=dots, pattern color=black]
(-0.4,-0.7)to[out= -90,in=180, looseness=1](0,-1.1)--
(0,-1.1)to[out= 0,in=-90, looseness=1](0.4,-0.7)--
(0.4,-0.7)to[out= -165,in=-15, looseness=1](-0.4,-0.7);
\filldraw[fill=white, color=blue, pattern=grid,  pattern color=blue]
(-0.4,-0.7)to[out= 90,in=180, looseness=1](0,-0.3)--
(0,-0.3)to[out= 0,in=90, looseness=1](0.4,-0.7)--
(0.4,-0.7)to[out= -165,in=-15, looseness=1](-0.4,-0.7);
\filldraw[fill=white, color=green(munsell), pattern=dots, pattern color=green(munsell)]
(0.825,-0.1)to[out= 10,in=-90, looseness=1](1.1,0.25)--
(1.1,0.25)to[out= 90,in=35, looseness=1](0.5,0.6)--
(0.5,0.6)to[out= -60,in=100, looseness=1](0.825,-0.1);
\filldraw[fill=white, color=yellow, pattern=grid, pattern color=yellow]
(0.825,-0.1)to[out= -170,in=-90, looseness=1](0.3,0.25)--
(0.3,0.25)to[out= 90,in=-145, looseness=1](0.5,0.6)--
(0.5,0.6)to[out= -60,in=100, looseness=1](0.825,-0.1);
\end{tikzpicture}\quad
\begin{tikzpicture}[scale=0.8]
\draw[black!50!white, dashed]
(-2,-1.85)to[out= 0,in=180, looseness=1] (2,-1.85)
(-2,1.85)to[out= 0,in=180, looseness=1] (2,1.85)
(-2,-1.85)to[out= 90,in=-90, looseness=1] (-2,1.85)
(2,-1.85)to[out= 90,in=-90, looseness=1] (2,1.85);
\draw[color=black, dashed, very thick]
(0.86,-0.5)to[out= 90,in=-30, looseness=1] (0,1)
(-0.86,-0.5)to[out= -30,in=-150, looseness=1] (0.86,-0.5);
\draw[color=black,scale=1,domain=-3.141: 3.141,
smooth,variable=\t,shift={(0,-0.7)},rotate=0]plot({0.4*sin(\t r)},
{0.4*cos(\t r)});
\fill[black](0,1) circle (1.7pt);
\fill[black](-0.85,-0.5) circle (1.7pt);
\fill[black](0.85,-0.5) circle (1.7pt);
\path[font=\normalsize]
(-0.86,-0.5)node[left]{$p_1$}
(0.86,-0.5)node[right]{$p_2$}
(0,1)node[above]{$p_3$};
\path[font=\small]
(-1.5,-1.85) node[above]{$D_2$};
\filldraw[fill=white, color=lavenderblue, pattern=north west lines, pattern color=lavenderblue]
(-0.4,-0.7)to[out= 90,in=180, looseness=1](0,-0.3)--
(0,-0.3)to[out= 0,in=90, looseness=1](0.4,-0.7)--
(0.4,-0.7)to[out= -165,in=-15, looseness=1](-0.4,-0.7);
\filldraw[fill=white,  color=blue, pattern=grid,  pattern color=blue]
(-0.4,-0.7)to[out= -90,in=180, looseness=1](0,-1.1)--
(0,-1.1)to[out= 0,in=-90, looseness=1](0.4,-0.7)--
(0.4,-0.7)to[out= -165,in=-15, looseness=1](-0.4,-0.7);
\filldraw[fill=white, color=tigerseye, pattern=north west lines, pattern color=tigerseye]
(0.825,-0.1)to[out= 10,in=-90, looseness=1](1.1,0.25)--
(1.1,0.25)to[out= 90,in=35, looseness=1](0.5,0.6)--
(0.5,0.6)to[out= -60,in=100, looseness=1](0.825,-0.1);
\filldraw[fill=white, color=green(munsell), pattern=dots, pattern color=green(munsell)]
(0.825,-0.1)to[out= -170,in=-90, looseness=1](0.3,0.25)--
(0.3,0.25)to[out= 90,in=-145, looseness=1](0.5,0.6)--
(0.5,0.6)to[out= -60,in=100, looseness=1](0.825,-0.1);
\end{tikzpicture}\quad
\begin{tikzpicture}[scale=0.8]
\draw[black!50!white, dashed]
(-2,-1.85)to[out= 0,in=180, looseness=1] (2,-1.85)
(-2,1.85)to[out= 0,in=180, looseness=1] (2,1.85)
(-2,-1.85)to[out= 90,in=-90, looseness=1] (-2,1.85)
(2,-1.85)to[out= 90,in=-90, looseness=1] (2,1.85);
\draw[color=black, dashed, very thick]
(0.86,-0.5)to[out= 90,in=-30, looseness=1] (0,1)
(-0.86,-0.5)to[out= -30,in=-150, looseness=1] (0.86,-0.5);
\draw[color=black,scale=1,domain=-3.141: 3.141,
smooth,variable=\t,shift={(0,-0.7)},rotate=0]plot({0.4*sin(\t r)},
{0.4*cos(\t r)});
\fill[black](0,1) circle (1.7pt);
\fill[black](-0.85,-0.5) circle (1.7pt);
\fill[black](0.85,-0.5) circle (1.7pt);
\path[font=\normalsize]
(-0.86,-0.5)node[left]{$p_1$}
(0.86,-0.5)node[right]{$p_2$}
(0,1)node[above]{$p_3$};
\path[font=\small]
(-1.5,-1.85) node[above]{$D_3$};
\filldraw[fill=white, pattern=dots]
(-0.4,-0.7)to[out= 90,in=180, looseness=1](0,-0.3)--
(0,-0.3)to[out= 0,in=90, looseness=1](0.4,-0.7)--
(0.4,-0.7)to[out= -165,in=-15, looseness=1](-0.4,-0.7);
\filldraw[fill=white,color=lavenderblue, pattern=north west lines, pattern color=lavenderblue]
(-0.4,-0.7)to[out= -90,in=180, looseness=1](0,-1.1)--
(0,-1.1)to[out= 0,in=-90, looseness=1](0.4,-0.7)--
(0.4,-0.7)to[out= -165,in=-15, looseness=1](-0.4,-0.7);
\filldraw[fill=white,color=yellow, pattern=grid, pattern color=yellow]
(0.825,-0.1)to[out= 10,in=-90, looseness=1](1.1,0.25)--
(1.1,0.25)to[out= 90,in=35, looseness=1](0.5,0.6)--
(0.5,0.6)to[out= -60,in=100, looseness=1](0.825,-0.1);
\filldraw[fill=white,color=tigerseye, pattern=north west lines, pattern color=tigerseye]
(0.825,-0.1)to[out= -170,in=-90, looseness=1](0.3,0.25)--
(0.3,0.25)to[out= 90,in=-145, looseness=1](0.5,0.6)--
(0.5,0.6)to[out= -60,in=100, looseness=1](0.825,-0.1);
\end{tikzpicture}
\end{center}
\caption{A closer look at the topology of the covering $Y$ of $\mathbb{R}^2\setminus\{p_1,p_2,p_3\}$. Each ball is represented with one texture and color.}
\end{figure}
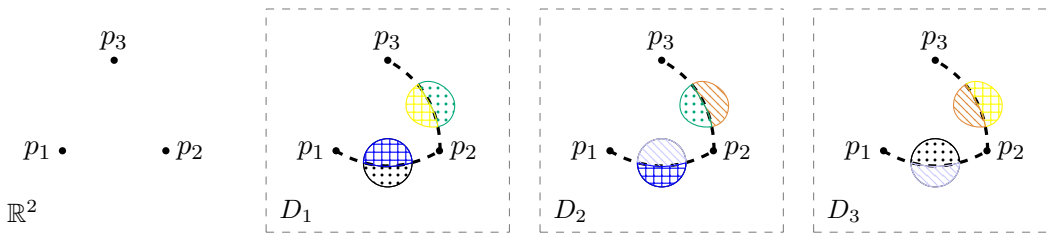

We remark that for technical reasons the construction in~\cite{cover}
requires the definition of a \emph{pair of cuts} joining the point $p_i$ with $p_{i+1}$. 
Then the equivalence relation is defined identifying the open sets enclosed by the pairs of cuts.

\begin{rem}\label{presc}
Given a function $f: Y \rightarrow\R^m$ 
it is possible to define the parametrizations of $f$ on the sheet $j$ 
as a function $f^j : M \rightarrow \R^m$ for every $j=1\ldots,n$ 
(see~\cite[Definition 2.5, Definition 2.6]{calicipi} for further details).

It is then possible to define functions $f$ (resp. sets $E$)
on the covering space $Y$ prescribing the parametrizations $f^j : M \rightarrow \R^m$ 
(resp. sets $E^j$) for every $j=1\ldots,n$. 
The set $E^j$ is the set determined by the parametrization of $\chi_E$ on the sheet $j$.
\end{rem}

We define now the class of sets $\mathscr{P}_{constr}(Y)$ that we will consider
to get the equivalence with the Steiner Problem. 
A set $E$ belongs to the space $\mathscr{P}_{constr}(Y)$
if it is a set of finite perimeter in $Y$,
for almost every $x$ in the base space there exists exactly one point $y$ of $E$
such that $p(y)=x$ and it satisfies a suitable boundary condition at infinity.

More precisely fixing an open, regular and bounded set 
$\Lambda\subset\mathbb{R}^2$ such that 
$\Sigma\subset \Lambda$ and $\mathrm{Conv}(S)\subset \Lambda$ (here $\mathrm{Conv}(S)$ denotes the convex envelope of $S$), we defined rigorously $\mathscr{P}_{constr}(Y)$ as follows:

\begin{dfnz}[Constrained sets]\label{zerounoconstr}
We denote by $\mathscr{P}_{constr}(Y)$ the space of
the sets of finite perimeter in $Y$ such that 
\begin{itemize}\label{zerouno}
\item [$i)$] $\sum_{p(y) = x} \chi_E(y) = 1\ \ $ for almost every $x\in M$,
\item [$ii)$] $\chi_{E^1}(x) = 1\ \  $ for every $x\in \mathbb{R}^2 \setminus \Lambda$.
\end{itemize}
\end{dfnz}

We look for 
\begin{equation}\label{minpro2}
\min \left\{P(E) : E\in\mathscr{P}_{constr}(Y)\right\}\,.
\end{equation}

\begin{rem}\label{independent}
Problem~\eqref{minpro2} does not depend on the choice of the cut $\Sigma$ 
in the definition of the covering space $Y$ (see \cite{cover}).
Moreover given $E_{min}$ a minimizer for~\eqref{minpro2} 
it is always possible to label the points $S$ in such a way 
that the cut $\Sigma$ does not intersect the projection of the reduced boundary of
$E_{\min}$ (see~\cite[Proposition 2.28]{calicipi}). 
From now on we always do this choice of the labeling of $S$.
\end{rem}

\begin{thm}
The Steiner Problem is equivalent to Problem~\eqref{minpro2}.
\end{thm}
\begin{proof}
See \cite[Theorem 2.30]{calicipi}.
\end{proof}

Once we have reduced the Steiner Problem to Problem~\eqref{minpro2},
a notion of calibration follows extremely naturally.

\begin{dfnz}[Calibration on coverings]\label{caliconvering}
Given $E\in\mathscr{P}_{constr}(Y)$,
a calibration for $E$
is an approximately regular vector field $\widetilde{\Phi} :Y\to\mathbb{R}^2$ (Definition \ref{approxi}) such that:
\begin{enumerate}
\item [(\textbf{1})] $\div\widetilde{\Phi} =0$;
\item [(\textbf{2})] $\vert \widetilde{\Phi} ^i (x) - \widetilde{\Phi} ^j (x)\vert \leq 2$ 
for every $i,j = 1,\ldots n$ and for every $x\in M$;
\item  [(\textbf{3})] $\int_{Y} \widetilde{\Phi}  \cdot D\chi_E=P(E)$.
\end{enumerate}
\end{dfnz}

As desired we have that 
if $\widetilde{\Phi} :Y\to\mathbb{R}^{2}$ is a calibration for $E$,
then $E$ is a minimizer of Problem~\eqref{minpro2}~\cite[Theorem 3.5]{calicipi}.
\medskip

We recall that we can reformulate the problem in terms of 
$BV$ functions with values in $\{0,1\}$:
we define $BV_{constr}(Y,\{0,1\})$ as the space of
functions $u \in BV(Y,\{0,1\})$ such that for almost every 
$x\in M $ it holds 
$\sum_{p(y) = x} u(y) = 1$
and $u^1(x) = 1$ for every $x\in \R^2 \setminus \Lambda$.
Then we minimize the total variation among functions in 
$BV_{constr}(Y,\{0,1\})$.

\subsection{Minimal partitions problem and paired calibrations \cite{lawmor}}\label{partizioni}

We provide here the definition of \emph{paired calibrations}~\cite{lawmor} 
of a minimal partition in the plane 
(see for example~\cite{chambollecremerspock} for this formulation).
To speak about minimal partitions and paired calibration 
we have to suppose that the points of 
$S$ lies on the boundary of an open smooth convex set $\Omega$.

We define
\begin{displaymath}
\mathcal{B} := \left\{u=(u_1,\ldots,u_n)\in BV(\Omega, \{0,1\}^n) \mbox{ such that } 
\sum_{i=1}^n u_i(x) = 1\ a.e.\mbox{ in } \Omega\right\}
\end{displaymath}
and a function $\overline{u} \in \mathcal{B}$ such that $\overline{u}_i = 1$ 
on the part of $\partial \Omega$ that connects $p_i$ with $p_{i+1}$. 

We then define the energy:
\begin{equation*}
\mathcal{E}(u):= \sum_{i=1}^n |Du_i|(\Omega)\,.
\end{equation*}

\begin{dfnz}\label{min}
A function $u_{min} \in \mathcal{B}$ is a minimizer 
for the partition problem if $u_{min} = \overline{u}$ on $\partial \Omega$ and
\begin{equation*}
\mathcal{E}(u_{min}) \leq \mathcal{E}(v)\,.
\end{equation*}
for every $v\in \mathcal{B}$ such that $v = \overline u$ on $\partial \Omega$.
\end{dfnz}

\begin{dfnz}[Paired calibration]\label{paired}
A \emph{paired calibration} for $u\in \mathcal{B}$ is a collection of 
$n$ approximately regular 
vector fields $\phi_1,\ldots,\phi_n: \Omega \rightarrow \R^2$ such that
\begin{itemize}
\item $\div \phi_i = 0$ \quad for every $i=1,\ldots,n$,
\item $|\phi_i - \phi_j|\leq 2$ \quad a.e. in $\Omega$ and for every $i,j =1,\ldots,n$,
\item $(\phi_i - \phi_j)\cdot \nu_{ij} = 2$ \quad $\Ha^1$--a.e. in  $J_{u_i} \cap J_{u_j}$ 
and for every $i,j =1,\ldots,n$,
\end{itemize}  
where $J_{u_i}$ is the jump set of the function $u_i$ and
$\nu_{ij}$ denotes the normal to $J_{u_i} \cap J_{u_j}$.
\end{dfnz} 

With this definition Morgan and Lawlor proved in~\cite{lawmor} that 
if there exists a paired calibration for a given  $u\in \mathcal{B}$,
then the latter is a minimizer of the minimal partition problem according to Definition~\ref{min}.
\medskip

Given $u=(u_1,\ldots,u_n)$ a minimizer for the partition problem 
the union of the jump sets of $u_i$ is a minimal Steiner network. 
Conversely given a minimal Steiner network $\mathcal{S}$ it is possible to construct 
$v =(v_1,\ldots,v_n) \in \mathcal{B}$ such that the union of the singular sets of $v_i$ 
is the network $\mathcal{S}$. Such a $v$ is a minimizer for the partition problem. 
Therefore Definition~\ref{paired} is a legitimate 
notion of calibration for the Steiner Problem as well.

\begin{rem}\label{pairedecovering}
Calibrations on coverings in Definition~\ref{caliconvering}
are a generalization of paired calibrations.
Indeed when the points $S$ lies on the boundary of a convex set
(the only case in which paired calibrations are defined)
the two notions are equivalent.

Suppose that the points of $S$ lie on the boundary of a convex set $\Omega$.
Then in the construction of $Y$ we can choose the cut $\Sigma$ 
outside $\Omega$.
Consider $u=(u_1,\ldots,u_n) \in \mathcal{B}$ a minimizer for the minimal partition problem 
and a paired calibration $(\phi_1,\ldots,\phi_n)$ for $u$.
Define then $\widetilde{u}\in BV_{constr}(Y,\{0,1\})$ 
prescribing the parametrization on each sheet of $Y$ as
\begin{equation*}
\widetilde{u}^i = u_{n+1-i} \qquad \mbox{for } i=1,\ldots,n\,.
\end{equation*}
Notice that with this choice 
$|D\widetilde u|(Y) = \mathcal{E}(u)$.
Define a vector field $\widetilde \Phi : Y \rightarrow \R^2$ 
prescribing its parametrizations on the sheets of the covering spaces (see Remark~\ref{presc}) as
\begin{equation*}
\widetilde{\Phi}^i=\phi_{n+1-i} \qquad \mbox{for } i=1,\ldots,n\,.
\end{equation*}
It is easy to check that $\widetilde{u}$ is a minimizer for Problem~\eqref{minpro2} and that
$\widetilde{\Phi}$ a calibration for $\widetilde{u}$ 
according to Definition~\ref{caliconvering}.

Similarly, given a calibration $\widetilde \Phi$ for $\widetilde u \in BV_{constr}(Y,\{0,1\})$ 
minimizer for Problem~\eqref{minpro2} one can construct a paired calibration for 
$u \in \mathcal{B}$ minimizer for the minimal partition problem. 
\end{rem}

\subsection{Currents with coefficients in $\mathbb{R}^n$  \cite{annalisaandrea}}\label{unocorrenti}

We briefly summarize here the theory of
currents with coefficients in $\mathbb{R}^n$ 
with the approach given in~\cite{annalisaandrea}.
The notion of
currents with coefficients in a group was introduced by W.~Fleming~\cite{fleming}. We mention also the 
the work of B.~White~\cite{white1, white2}.\\

Consider the normed space $(\mathbb{R}^n, \Vert\cdot\Vert)$ and 
denote by  $\Vert\cdot\Vert_{\ast}$ the dual norm.
For $k=0,1,2$ we call $\Lambda_k(\mathbb{R}^2)$ the space of $k$--vectors in $\R^2$.

\begin{dfnz}[$k$--covector with values in $\mathbb{R}^n$]
A $k$--covector with values in $\mathbb{R}^n$ is a linear map from 
$\Lambda_k(\mathbb{R}^2)$ to $\mathbb{R}^n$.
We denote by $\Lambda^k_{n}(\mathbb{R}^2)$
the space of $k$--covectors with values in $\mathbb{R}^n$.
\end{dfnz}

We define the comass norm of a covector 
$\omega\in \Lambda^k_{n}(\mathbb{R}^2)$  as
\begin{equation*}
\vert \omega\vert_{com}:=\sup\left\lbrace
\Vert \omega(\tau)\Vert_\ast\,:\;\tau\in \Lambda_k(\mathbb{R}^2)\;\text{with}\,
\vert \tau\vert\leq1 \mbox{ and } \tau \mbox{ simple}
\right\rbrace\,.
\end{equation*}

Then the $k$--forms with values in $\R^n$ are defined as the vector fields  
$\omega\in C^\infty_c(\mathbb{R}^2, \Lambda^k_{n}(\mathbb{R}^2))$ 
and their comass is given by 
\begin{equation*}
\Vert\omega\Vert_{com}:=\sup_{x\in\mathbb{R}^2} \vert\omega(x)\vert_{com}\,.
\end{equation*}

\begin{rem}
Notice that the definition of the space 
$C^\infty_c(\mathbb{R}^2, \Lambda^k_{n}(\mathbb{R}^2))$
is equivalent to the one presented in~\cite{annalisaandrea}. 
Indeed they consider $k$--covectors $\omega$ defined as bilinear maps
\begin{equation*}
\omega : \Lambda_k(\R^2) \times \R^n \rightarrow \R\,,
\end{equation*}
that can be seen as $k$--covectors with values in $(\R^n)'$.
\end{rem}

Thanks to the just defined notions we are able to introduce the definition 
of $k$--current with coefficients in $\R^n$.
\begin{dfnz}[$k$--current with coefficients in $\mathbb{R}^n$]
A $k$--current with coefficients in $\mathbb{R}^n$
is a linear and continuous map
\begin{equation*}
T:C^\infty_c(\mathbb{R}^2, \Lambda^k_{n}(\mathbb{R}^2))\to \mathbb{R}\,.
\end{equation*}
\end{dfnz}

The \emph{boundary} of a 
$k$--current $T$ with coefficients in $\mathbb{R}^n$  is a $(k-1)$--current
defined as
\begin{equation*}
\partial T(\omega):=-T(d\omega)\,,
\end{equation*}
where $d\omega$ is defined component--wise. 

\begin{dfnz}[Mass]
Given $T$ a $k$--current with coefficients in $\mathbb{R}^n$ its mass is
\begin{equation*}
\mathbb{M}(T):=\sup\left\lbrace T(\omega)\,:\;\omega\in 
C^\infty_c(\mathbb{R}^2, \Lambda^k_{n}(\mathbb{R}^2))
\;\text{with}\,\Vert\omega\Vert_{com}\leq 1
\right\rbrace\,.
\end{equation*}
\end{dfnz}

A $k$--current $T$ with coefficients in $\mathbb{R}^n$ is said to be \emph{normal}
if $\mathbb{M}(T)<\infty$ and $\mathbb{M}(\partial T)<\infty$.

\begin{dfnz}[$1$--rectifiable current with coefficients in $\mathbb{Z}^n$]
Given $\Sigma$ a $1$--rectifiable set oriented by $\tau\in\Lambda_1(\R^2)$, simple, such that
$\vert\tau(x)\vert=1$ for a.e. $x\in \Sigma$ and $\theta:\Sigma\to\mathbb{Z}^{n}$ in $L^1(\Ha^1)$,
a $1$--current $T$ is rectifiable  with coefficients in $\Z^n$  
if admits the following representation:
\begin{equation*}
T(\omega)=\int_{\Sigma} \left\langle
\omega(x)(\tau(x)), \theta(x)\right\rangle\,\mathrm{d}\mathcal{H}^1\,.
\end{equation*}
A $1$--rectifiable current with coefficients in $\mathbb{Z}^n$
will be denoted by the triple $T=[\Sigma, \tau, \theta]$.
\end{dfnz}

Notice that if $T=[\Sigma, \tau, \theta]$ is a 
$1$--rectifiable current with coefficients in $\mathbb{Z}^n$
one can write its mass as 
\begin{equation*}
\mathbb{M}(T)=\int_{\Sigma}\Vert \theta(x)\Vert\,\mathrm{d}\mathcal{H}^1\,.
\end{equation*}

\begin{rem}
The space of $1$--covector with values in $\R^n$ can be identified 
with the set of matrices $M^{n\times 2}(\R)$. 
In what follows we will assume this identification 
and we will denote the set of $1$--forms by $C_c^\infty(\mathbb{R}^2,M^{n\times 2}(\mathbb{R}))$.
Moreover given $\omega \in C_c^\infty(\mathbb{R}^2,M^{n\times 2}(\mathbb{R}))$ we write it as
\begin{equation*}
\omega=
\begin{bmatrix}
\omega_1(x) \\
\vdots \\
\omega_n(x) \\
\end{bmatrix}\,,
\end{equation*}
where $\omega_i : \R^2 \rightarrow \R^2$. Notice that $\omega_i(x)$ is a canonical $1$-form, hence its differential can be identified (by the canonical Hodge dual) as
\begin{displaymath}
d\omega_i = \frac{\partial \omega}{\partial x_2} - \frac{\partial \omega}{\partial x_1} = \div \omega_i^\perp   
\end{displaymath}
and therefore we can define $d\omega$ as
\begin{equation*}
d\omega=
\begin{bmatrix}
\div \omega^\perp_1 \\
\vdots \\
\div \omega^\perp_n \\
\end{bmatrix}\,.
\end{equation*}
\end{rem}

Let $(g_i)_{i=1,\ldots,n-1}$ be the canonical base of $\R^{n-1}$. 
Define $g_n = -\sum_{i=1}^{n-1} g_i$.

Given $B = g_1 \delta_{p_1} + \ldots + g_n \delta_{p_n}$ 
we consider the following minimization problem:
\begin{equation}\label{minprocurrents}
\inf \left\{\mathbb{M}(T)\  : \  T\mbox{ is a }  1-\text{rectifiable 
currents with coefficients if } \Z^{n-1} ,\ \partial T = B\right\}\,. 
\end{equation}

To have the equivalence between Problem~\eqref{minprocurrents}
and the Steiner Problem~\eqref{ste} the choice 
of the norm of $\mathbb{R}^{n-1}$ plays an important role. 
Indeed given $\mathcal{I}$ any subset of $\{1,\ldots,n-1\}$ it is required in \cite{annalisaandrea} that
\begin{equation}\label{prope}
\left\lVert \sum_{i\in \mathcal{I}} g_i\right\rVert =1\,.
\end{equation}

\begin{thm}\label{zu}
Choosing a norm satisfying \eqref{prope}, the Steiner Problem is equivalent to Problem \eqref{minprocurrents}.
\end{thm}

The notion of calibration associated to the mass minimization problem \eqref{minprocurrents} introduced in \cite{annalisaandrea} is the following:

\begin{dfnz}[Calibration for $1$--rectifiable currents]\label{calicurrents}
Let $T=[\Sigma, \tau, \theta]$ be a 
$1$--rectifiable current with coefficients in $\mathbb{Z}^{n-1}$ and 
$\Phi\in C_c^\infty(\mathbb{R}^2,M^{{n-1}\times 2}(\mathbb{R}))$. 
Then 
$\Phi$ is a calibration for $T$ if
\begin{itemize}
\item[(i)] $d\Phi = 0$;
\item[(ii)] $\|\Phi\|_{com} \leq 1$;
\item[(iii)]  $\langle \Phi(x)\tau(x),\theta(x)\rangle 
= \|\theta(x)\|$ for $\Ha^1$-a.e. $x \in \Sigma$.
\end{itemize}
\end{dfnz}

If $\Phi\in C_c^\infty(\mathbb{R}^2,M^{{n-1}\times 2}(\mathbb{R}))$ is a calibration
for $T=[\Sigma, \tau, \theta]$ a
$1$--rectifiable current with coefficients in $\mathbb{Z}^{n-1}$,
then $T$ is a minimizer of Problem~\eqref{minprocurrents}.
To be more precise $T$ is a minimizer among normal currents 
with coefficients in $\mathbb{R}^{n-1}$ \cite{annalisaandrea}.

\begin{rem}
In Proposition~\ref{approximately} in appendix we prove that is possible 
to weaken the regularity of the calibration $\Phi$ and consider 
$\Phi : \R^2 \rightarrow M^{{n-1}\times 2}(\mathbb{R})$ such that each row 
is an approximately regular vector field (see also \cite{annalisaandrea} for a definition of calibration with weaker regularity assumptions of the vector fields).
In the next section we assume implicitly that $\Phi$ is approximately regular. 
\end{rem}

\section{Relations among  the different notions of
calibrations}\label{equivalence}

We have already discussed the equivalence between paired calibrations
and calibrations on coverings (see Remark~\ref{pairedecovering}).
We focus now on the relation with Definition~\ref{calicurrents}.

Definition~\ref{calicurrents} is dependent on the norm of  $\R^n$.
Define $\Vert \cdot\Vert _{\flat}$ as
\begin{equation*}
\Vert x\Vert _{\flat}:=\sup_{x_i>0} x_i - \inf_{x_i\leq 0} x_i
\end{equation*}
for every $x\in \mathbb{R}^n$.
This is the norm considered by Marchese and Massaccesi~\cite{annalisaandrea} and in particular it satisfies property \eqref{prope}.
In~\cite{annalisaandrea}  it is also proved that
the dual norm $\Vert \cdot\Vert_{\flat,\ast}$ can be characterized as
follows:
\begin{equation}\label{dualnorm}
\Vert x\Vert _{\flat,\ast}= \max\left\{
\sum_{x_i > 0} x_i,
\sum_{x_i \leq 0}|x_i|
\right\}\,.
\end{equation}

\medskip

\textbf{From calibrations for currents to calibration on coverings}

\medskip 

From here on we endow $\mathbb{R}^n$
with $\Vert \cdot\Vert=\Vert \cdot\Vert _{\flat}$.
With this choice,
we show that if there exists a calibration for a
$1$--rectifiable current with coefficients in $\mathbb{Z}^{n-1}$,
then there exists a calibration for $E\in\mathscr{P}_{constr}$
minimizer for Problem~\eqref{minpro2}. 

\begin{lemma}\label{construction}
Given $S=\{p_1,\ldots,p_n\}$ with the points $p_i$
lying on the boundary of a convex set $\Omega$
labelled in an anticlockwise sense and $u=(u_1,\ldots,u_n)$
a competitor of the minimal partition problem, it is possible to
construct  a $1$--rectifiable current $T=[\Sigma, \tau, \theta]$
with coefficients in $\mathbb{Z}^{n-1}$ such that
$2\mathbb{M}(T)=\mathcal{E}(u)$,
$\partial T=g_1\delta_{p_1}+\ldots g_n\delta_{p_n}$
and for $\Ha^1$--a.e. $x\in J_{u_i} \cap J_{u_j}$
\begin{equation}\label{summ}
\theta(x) = \sum_{k=i}^{j-1} g_k\,.
\end{equation}
\end{lemma}
\begin{proof}
For $i=1,\ldots,n$ let $A_i$ be the
phases of the partition induced by $u=(u_1,\ldots,u_n)$, that is
$u_i=\chi_{A_i}$.
Notice that for every $i\in\{1,\ldots,n\}$
the set $\partial^\ast A_i$ is a $1$--rectifiable set in $\R^2$
with tangent $\tau_i$ almost everywhere and it joins the points $p_i$ and
$p_{i+1}$.
For $i\in\{1,\ldots,n\}$
we define  $T_i=[\partial^\ast A_i, \tau_i, a_i]$,
where the multiplicities $a_i$ are chosen in such a way that
 $a_{i} - a_{i+1} = g_i$ for $i=1,\ldots, n-1$.
 Then
 for every $i,j= 1,\ldots, n$
\begin{equation}\label{molt}
a_i - a_j = \sum_{k=i}^{j-1} g_k\,.
\end{equation}
We set
\begin{equation*}
T = \sum_{i=1}^n T_i\,.
\end{equation*}
Denoting by $\theta_{T}$ the multiplicity of $T$,
by construction $\theta_{T}(x) = a_i - a_j$ for $\Ha^1$--a.e.
$x \in \partial^\ast A^i \cap \partial^\ast A^j$ that thanks to \eqref{molt} gives \eqref{summ}.
Moreover as $\partial T_{i} = a_i (\delta_{p_{i}} - \delta_{p_{i+1}})$
(with the convention that $p_{n+1} = p_1$).
We infer
\begin{eqnarray*}
\partial T &=& \sum_{i=1}^n \partial T_i
= \sum_{i=1}^n a_{i}(\delta_{p_{i}} - \delta_{p_{i+1}}) = \sum_{i=1}^n
a_{i}\delta_{p_{i}} - \sum_{i=1}^n a_{i}\delta_{p_{i+1}} \\
&=&  \sum_{i=1}^{n} a_{i}\delta_{p_{i}} - \sum_{i=2}^{n+1}
a_{i-1}\delta_{p_i} = (a_1 - a_n)\delta_{p_1}
+ \sum_{i=2}^{n} (a_i - a_{i+1}) \delta_{p_i}\\
&=& \sum_{i=1}^n  g_i \delta_{p_i}\,.
\end{eqnarray*}
\end{proof}

\begin{thm}\label{dacalicurrentacalicovering}
Given $S=\{p_1,\ldots,p_n\}$ lying on the boundary of a convex set $\Omega$,
let $E\in\mathscr{P}_{constr}$ be a
minimizer for Problem~\eqref{minpro2}.
Let $\Phi$ be a calibration
according to Definition~\ref{calicurrents}
for $T=[\Sigma,\tau,\theta]$ a $1$--rectifiable current with
coefficients in
$\mathbb{Z}^{n-1}$ minimizer of Problem~\eqref{minprocurrents}.
Then there exists $\widetilde{\Phi}$ a calibration for $E$
(according to Definition~\ref{caliconvering}).
\end{thm}
\begin{proof}
We label the $n$ points of $S$ in an anticlockwise sense.
By choosing the cuts $\Sigma\supset \Omega$ calibrations on coverings (Definition~\ref{caliconvering}) reduce to
paired calibrations (Definition~\ref{paired}).
Hence we are looking for a collection $\widetilde{\Phi}$
of $n$ vector fields  $\widetilde{\Phi}^i:\Omega\to \mathbb{R}^2$
for  $u=(u_1,\ldots,u_n)\in\mathcal{B}$ where $u_{n+1-i}=\chi_{E^i}$ (see
Remark \ref{pairedecovering}).

Let $\Phi_i$ be the $n-1$ rows of the matrix $\Phi$.
We claim that the collection
of $n$ vector fields $\widetilde{\Phi}^i$ defined by 
\begin{equation}\label{relation}
\widetilde{\Phi}^i-\widetilde{\Phi}^{i+1}=2\Phi_i^\perp \quad \mbox{for }
i=1,\ldots, n-1
\end{equation}
is a calibration for $E$. Notice that from~\eqref{relation} we deduce that
\begin{equation*}
\widetilde{\Phi}^{n}-\widetilde{\Phi}^1
=-2\sum_{i=1}^{n-1}\Phi_i^\perp\,.
\end{equation*}
We have to show that $\widetilde\Phi$ satisfies conditions (\textbf{1}),
(\textbf{2}), (\textbf{3}) of Definition \ref{caliconvering}.
\begin{itemize}
\item[(\textbf{1})] The divergence of $\widetilde\Phi$
is automatically zero, because $\div \widetilde \Phi^i = 0$ in $\Omega$
for every $i=1,\ldots,n$ (notice that we have taken the cut $\Sigma$
outside $\Omega$).
\item[(\textbf{2})] By Definition~\ref{calicurrents} it holds
that $\|\Phi\|_{com} \leq 1$, hence for every $x\in\mathbb{R}^2$ we have
\begin{equation*}
\sup\left\lbrace\Vert\Phi(x)\tau\Vert_{\flat, \ast}
\;\Big\vert\, \tau\in\mathbb{R}^2,\;\vert\tau\vert \leq 1\right\rbrace\leq
1\,.
\end{equation*}
Writing the $n$ components of
$\Phi(x)\tau$ as $\left\langle\Phi_1(x),\tau\right\rangle,
\ldots, \left\langle\Phi_n(x),\tau\right\rangle$ and
using~\eqref{dualnorm},
for every $\tau \in \R^2$ such that $|\tau| \leq 1$  and $i,j =
1,\ldots,n$ with $i \leq j-1$ we obtain
\begin{align*}
1 \geq &\Vert\Phi(x)\tau\Vert_{\flat, \ast}
=\max\left\lbrace
\sum_{\left\langle\Phi_k,\tau\right\rangle>0}
\left\langle\Phi_k,\tau\right\rangle,
-\sum_{\left\langle\Phi_k,\tau\right\rangle<0}\left\langle\Phi_k,\tau\right\rangle
\right\rbrace
\\
& \geq  \left\lvert \sum_{k=1}^{n-1}\left\langle\Phi_k,\tau\right\rangle
\right\rvert
\geq \left\lvert \left\langle
\sum_{k=i}^{j-1} \Phi_{k}, \tau
\right\rangle\right\rvert\,.
\end{align*}
Therefore
\begin{equation*}
\left|\sum_{k=i}^{j-1} \Phi_k^\perp\right| \leq 1\,.
\end{equation*}
Notice that from \eqref{relation} we obtain that for every
$i,j\in\{1,\ldots,n\}$
\begin{equation*}
\vert \widetilde{\Phi}^i-\widetilde{\Phi}^j\vert
=2\left\vert \sum_{k=i}^{j-1}\Phi_k^\perp\right\vert\,.
\end{equation*}
Hence
for every $i,j\in\{1,\ldots,n\}$ condition
$\vert \widetilde{\Phi}^i(x)-\widetilde{\Phi}^j(x)\vert\leq 2$ is
fulfilled for every $x\in \Omega$.
\item[(\textbf{3})]
We can apply
the construction of Lemma~\ref{construction} with  $E^i=A_{n+1-i}$ to
produce a $1$--rectifiable current $\overline{T}$
with coefficients in $\Z^{n-1}$ such that
and $\partial \overline{T} = \partial T$ and
 $2\mathbb{M}(\overline{T}) = P(E)$.
 Moreover
thanks to the fact that $E$ a minimizer for~\eqref{minpro2} and $T$
for~\eqref{minprocurrents}
we get
\begin{equation*}
2\mathbb{M}(\overline{T}) =
P(E)=2\mathcal{H}^1(\mathcal{S})=2\mathbb{M}(T)\,,
\end{equation*}
where $\mathcal{S}$ is a minimizer for the Steiner Problem~\eqref{ste}.
The current $\overline{T}$ has the same boundary and the same mass of $T$,
then it  is a minimizer for the Problem~\eqref{minprocurrents} as well.
Therefore $\Phi$ is a calibration also for $\overline{T}$.
Then we have that $\Ha^1$--a.e.
$x\in \overline{p(\partial^\ast E)}$
\begin{displaymath}
\langle \Phi \tau, \theta_{\overline{T}} \rangle = 1\,.
\end{displaymath}
Using~\eqref{summ}, for $\Ha^1$--a.e. $x \in \partial^\ast A^i \cap
\partial^\ast A^j$
the previous equation reads as
\begin{equation*}
1 = \langle \Phi \tau, \sum_{k=i}^{j-1} g_k \rangle = \sum_{k=i}^{j-1}
\langle \Phi \tau,  g_k \rangle =
\sum_{k=i}^{j-1} \langle \Phi_k, \tau\rangle=
\sum_{k=i}^{j-1} \langle \Phi^\perp_k, \nu_{ij}\rangle =
\frac{1}{2}\langle \widetilde{\Phi}^i  - \widetilde{\Phi}^j, \nu_{ij}
\rangle\,,
\end{equation*}
that is the third condition of the paired calibration, that in our setting
is equivalent to (\textbf{3}) of Definition \ref{caliconvering}.
\end{itemize}
 \end{proof}

When the points of $S$ do not lie on the boundary of a convex set,
we cannot take advantage of the equivalence between calibrations on coverings and 
paired calibrations 
(that are not defined if the points of $S$ are not on the boundary of a convex set).
Indeed in this case we look for a unique vector field $\widetilde{\Phi}:Y\to\mathbb{R}^2$
defined on the whole space $Y$ and 
satisfying the requirements of Definition~\ref{caliconvering}.
As one can guess from Step \textbf{(3)} in the proof of
Theorem~\ref{dacalicurrentacalicovering},
relations~\eqref{relation}
has to be satisfied locally around the jumps.
A clue of this fact is
the \emph{local} equivalence between the minimal partition problem
and Problem~\eqref{minpro2} that suggests a local
equivalence between calibrations on coverings and paired calibrations.
Thanks to Remark~\ref{presc}
once defined $\widetilde{\Phi}^i$
in each sheet
one can construct $\widetilde{\Phi}$,
but in doing such an extension/identification
procedure it  is not guaranteed that the divergence of 
$\widetilde{\Phi}$ is zero.
It seems to us that the existence of a calibration $\widetilde{\Phi}$
is plausible, but the extension of the field
has to be treated case by case.
At the moment we do not have a procedure to construct $\widetilde{\Phi}$  globally.

\medskip

\textbf{From calibrations on coverings to calibrations for currents}

\medskip

Given a calibration for $E\in\mathscr{P}_{constr}$ 
minimizer for Problem~\eqref{minpro2}
we want now to construct a calibration for $T$,
a $1$--rectifiable current with coefficients in $\mathbb{Z}^{n-1}$
minimizer for Problem~\eqref{minprocurrents}.

Notice that given any competitor $\overline{T}=[\overline{\Sigma},\overline{\tau},\overline{\theta}]$, testing condition ii) of Definition~\ref{calicurrents} on $\overline{T}$ reduces to show that 
\begin{equation*}
\left\langle\Phi(x)\overline{\tau}(x),\overline{\theta}(x)\right\rangle
\leq \Vert \overline{\theta}(x)\Vert_\flat \quad \mbox{for }\mathcal{H}^1-a.e \ x\in\overline{\Sigma}\, .
\end{equation*}
Moreover it suffices to evaluate
$\left\langle\Phi(x)\overline{\tau}(x),\cdot \right\rangle$
on the extremal points of the unit ball of the norm $\Vert\cdot\Vert_\flat$
that are $P_{\mathcal{I}} = \sum_{i\in\mathcal{I}} g_i$ for every $\mathcal{I}\subset\{1,\ldots,n-1\}$
(see~\cite[Example 3.4]{annalisaandrea}).
Hence proving Condition  ii) reduces to verify $2^{n-1}-1$ inequalities.
On the other hand Condition (\textbf{2}) of Definition~\ref{caliconvering}
requires to verify $\frac{n(n-1)}{2}$ inequalities.
Apart from the case of $2$ and $3$ points, 
Condition (\textbf{2}) of Definition~\ref{caliconvering} is weaker than
Condition ii) of Definition~\ref{calicurrents}. 
Hence in general one cannot construct a calibration for $T$
starting from a calibration for $E$.

\medskip

To restore an equivalence result we slightly change Problem~\eqref{minprocurrents}.

Define a norm $\Vert\cdot\Vert_\natural$ on $\mathbb{R}^{n-1}$
characterized by the property that its unit ball is the smallest such that
\begin{equation*}
\left\Vert \sum_{k=i}^{j-1} g_i \right\Vert_\natural=1 \quad\text{with}\ i\leq j-1\ ,\ i,j \in \{1,\ldots,n\}\,.
\end{equation*}

For $n=4$ (in this case the admissible coefficients are $g_1,g_2$ and $g_3$) the unit ball of the norm $\Vert\cdot\Vert_{\natural}$ is depicted in Figure~\ref{norm}. 
Notice that if we consider the norm $\Vert\cdot\Vert_\flat$, the mass of all curves appearing in Figure~\ref{exampleA} coincides with the length.
If instead we use the norm $\Vert\cdot\Vert_{\natural}$, the mass
of the curve with multiplicity $g_1+g_3$ is strictly bigger than its length. This is a still natural choice if we want to prove an equivalence with the Steiner problem as in Theorem \ref{zu} for the norm $\Vert\cdot\Vert_{\natural}$. Indeed for a specific labelling of the points,  the curve with multiplicity $g_1 + g_3$ has to lie outside the convex envelope of $p_1,\ldots,p_4$ and therefore the competitor on the rightmost of Figure~\ref{exampleA} cannot be a minimizer for the Steiner problem.

From now on we write either $\mathbb{M}_\natural$ or  $\mathbb{M}_\flat$
to distinguish when the mass is computed using 
either the norm $\Vert\cdot\Vert_\natural$ or $\Vert\cdot\Vert_\flat$. 


\begin{figure}[h!]
\begin{center}
\begin{tikzpicture}[scale=0.7]
\draw[black!50!white]
(0,0)--(0,5)
(0,0)--(-2.5,-4.33)
(0,0)--(4.33, -2.5);
(0,0)--(2.5,4.33)
(0,0)--(0,-5)
(0,0)--(-4.33, 2.5);
\draw[black]
(-1.5,-2.6)--(-1.5,0.4) %
(2.6, -1.5)--(2.6, 1.5) %
(0,3)--(-1.5,0.4)
(0,3)--(2.6, 1.5)
(-1.5,0.4)--(1.1,-1.09)
(2.6, 1.5)--(1.1,-1.09)
(1.1,-1.09)--(1.1,-4.09)
(-1.5,-2.6)--(1.1,-4.09)
(2.6, -1.5)--(1.1,-4.09);
\draw[black, dashed]
(0,-3)--(1.1,-4.09)%
(0,-3)--(1.5,-0.4)%
(0,-3)--(-2.6, -1.5)%
(1.5,-0.4)--(1.5,2.6)
(1.5,-0.4)--(-1.1,1.09)
(-2.6, -1.5)--(-1.1,1.09)
(-1.1,1.09)--(-1.1,4.09)
(2.6, -1.5)--(1.5,-0.4);
\draw[black]
(1.5,2.6)--(-1.1,4.09)
(-2.6,1.5)--(-1.1,4.09)
(1.5,2.6)--(2.6, 1.5)
(-2.6, -1.5)--(-2.6,1.5)
(0,3)--(-1.1,4.09) %
(-2.6, 1.5)--(-1.5,0.4)
(-2.6, -1.5)--(-1.5,-2.6);
\fill[black](0,-3)circle (1.7pt);
\fill[black](-2.6, 1.5)circle (1.7pt);
\fill[black](1.5,2.6)circle (1.7pt);
\fill[black](-2.6, -1.5)circle (1.7pt);
\fill[black](-1.1,1.09)circle (1.7pt);
\fill[black]((-1.1,4.09)circle (1.7pt);
\fill[black](0,3)circle (1.7pt);
\fill[black](2.6, 1.5)circle (1.7pt);
\fill[black](-1.5,0.4)circle (1.7pt);
\fill[black](1.1,-1.09)circle (1.7pt);
\fill[black](-1.5,-2.6)circle (1.7pt);
\fill[black](2.6, -1.5)circle (1.7pt);
\fill[black](1.1,-4.09)circle (1.7pt);
\path[font=\tiny]
(0.1,2.85)node[below]{$g_2$}
(2.6, 1.5)node[right]{$g_1+g_2$}
(-1.5,0.4)node[right]{$g_2+g_3$}
(1.1,-1.09)node[left]{$g_1+g_2+g_3$}
(-1.5,-2.6)node[left]{$g_3$}
(1.1,-4.09)node[below]{$g_1+g_3$}
(2.6, -1.4)node[right]{$g_1$};
\end{tikzpicture}\quad\quad\quad
\begin{tikzpicture}[scale=0.7]
\draw[black!50!white]
(0,0)--(0,5)
(0,0)--(-2.5,-4.33)
(0,0)--(4.33, -2.5);
\draw[black]
(-1.5,-2.6)--(-1.5,0.4)
(2.6, -1.5)--(2.6, 1.5)
(0,3)--(-1.5,0.4)
(0,3)--(2.6, 1.5)
(-1.5,0.4)--(1.1,-1.09)
(2.6, 1.5)--(1.1,-1.09)
(1.1,-1.09)--(-1.5,-2.6)
(1.1,-1.09)--(2.6, -1.5);
\draw[black, dashed]
(0,-3)--(1.5,-0.4)
(0,-3)--(-2.6, -1.5)
(1.5,-0.4)--(1.5,2.6)
(1.5,-0.4)--(-1.1,1.09)
(-2.6, -1.5)--(-1.1,1.09)
(-2.6,1.5)--(-1.1,1.09)
(1.5,2.6)--(-1.1,1.09)
(2.6, -1.5)--(1.5,-0.4);
\draw[black]
(1.5,2.6)--(0,3)
(-2.6, 1.5)--(0,3)
(-2.6, -1.5)--(-2.6,1.5)
(1.5,2.6)--(2.6, 1.5)
(-1.5,-2.6)--(0,-3)
(2.6, -1.5)--(0,-3)
(-2.6, 1.5)--(-1.5,0.4)
(-2.6, -1.5)--(-1.5,-2.6);
\fill[black](0,-3)circle (1.7pt);
\fill[black](-2.6, 1.5)circle (1.7pt);
\fill[black](1.5,2.6)circle (1.7pt);
\fill[black](-2.6, -1.5)circle (1.7pt);
\fill[black](-1.1,1.09)circle (1.7pt);
\fill[black](0,3)circle (1.7pt);
\fill[black](2.6, 1.5)circle (1.7pt);
\fill[black](-1.5,0.4)circle (1.7pt);
\fill[black](1.1,-1.09)circle (1.7pt);
\fill[black](-1.5,-2.6)circle (1.7pt);
\fill[black](2.6, -1.5)circle (1.7pt);
\path[font=\tiny, white]
(1.1,-4.09)node[below]{$g_1+g_3$};
\path[font=\tiny]
(0.1,2.85)node[below]{$g_2$}
(2.6, 1.5)node[right]{$g_1+g_2$}
(-1.5,0.4)node[right]{$g_2+g_3$}
(1.1,-1.09)node[left]{$g_1+g_2+g_3$}
(-1.5,-2.6)node[left]{$g_3$}
(2.6, -1.4)node[right]{$g_1$};
\end{tikzpicture}
\end{center}
\caption{Left: the unit ball of the norm $\Vert\cdot\Vert_\flat$. Right: the unit ball of the
norm $\Vert\cdot\Vert_\natural$.}\label{norm}
\end{figure}
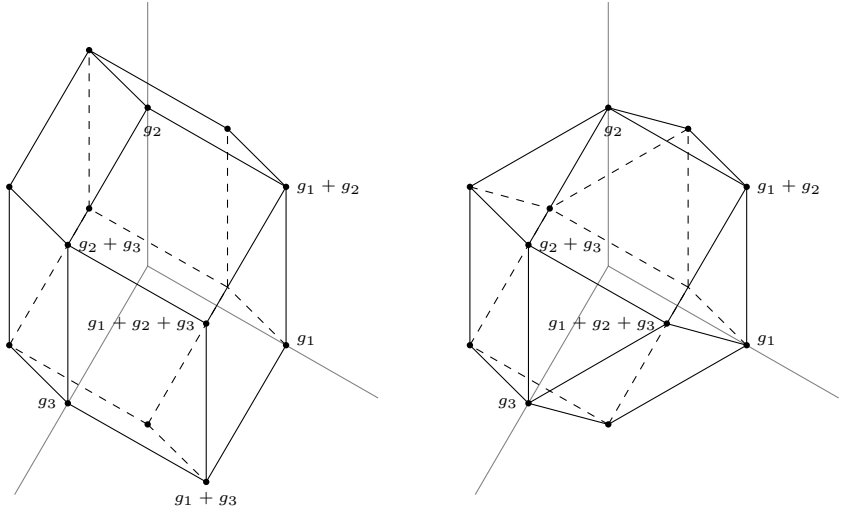


\begin{lemma}\label{dellefoglie}
There exists a permutation $\sigma$ of the labelling of the points of $S$ such that defining 
\begin{equation*}
B_\sigma = \sum_{i=1}^{n-1} g_{\sigma(i)} \delta_{p_{\sigma(i)}}
\end{equation*}
the problem
\begin{align}\label{sigmaminprocurrents}
\inf \left\{\mathbb{M}_\natural(T)\  :\  T=[\Sigma,\tau,\theta]  \mbox{ is a } 1-\text{rectifiable 
currents with coefficients in } \Z^{n-1}, \, \partial T = B_\sigma\right\}
\end{align}
is equivalent to the Steiner Problem.
\end{lemma}

Define
\begin{equation*}
\mathcal{G}:=\left\{\sum_{k=i}^{j-1}g_k : i,j = 1,\ldots,n \mbox{ and } i \leq j -1\right\}
\end{equation*}

and notice that by the definition of $\Vert\cdot\Vert_\natural$ one has 
\begin{equation}\label{minor}
\Vert\theta\Vert_\natural \geq \Vert\theta\Vert_\flat \quad \forall \theta \in \R^{n-1} \qquad \mbox{and} \qquad \Vert\theta\Vert_\natural = \Vert\theta\Vert_\flat=1 \quad\forall \theta \in \mathcal{G}.
\end{equation}

We obtain the definition of calibration for Problem~\eqref{sigmaminprocurrents} simply repeating
Definition~\ref{calicurrents} 
replacing $\Vert\cdot\Vert_\flat$ by $\Vert\cdot\Vert_\natural$. 
Clearly if $\Phi$ is a calibration for  $T_\sigma$, then 
$T_\sigma$ is a minimizer for  Problem~\eqref{sigmaminprocurrents} 
in its homology class.
We postpone the proof of Lemma \ref{dellefoglie} and we state the main result.

\begin{thm}
Given $S=\{p_1,\ldots,p_n\}$,
let $E\in\mathscr{P}_{constr}$ be a
minimizer for Problem~\eqref{minpro2}
and  $T=[\Sigma,\tau,\theta]$ be a $1$--rectifiable current with coefficients in 
$\mathbb{Z}^{n-1}$ minimizer of Problem~\eqref{sigmaminprocurrents}.
Suppose that there exists $\widetilde{\Phi}$ calibration for $E$
according to Definition~\ref{caliconvering}, then 
there exists $\Phi$  calibration for $T_\sigma$
according to Definition~\ref{calicurrents} (where we consider 
$\Vert\cdot\Vert=\Vert\cdot\Vert_\natural$).
\end{thm}
\begin{proof}
For simplicity we suppose that the $n$ points lie on the boundary of a 
convex set $\Omega$ and the cuts $\Sigma$ are chosen such that $\Sigma \supset \Omega$. Hence $\widetilde{\Phi}:Y\to\mathbb{R}^2$ reduces to 
a paired calibration: a collection of
$n$ approximately regular vector fields $\widetilde{\Phi}^i: \Omega \to\mathbb{R}^2$ satisfying the conditions of Definition \ref{paired}.
We define $\Phi$ calibration for $T_\sigma$ as the matrix 
whose $n-1$ rows satisfy
\begin{equation*}
\Phi^\perp_i=\frac{1}{2}\left(\widetilde{\Phi}^i-\widetilde{\Phi}^{i+1}\right)
\quad\text{for}\;i=1,\ldots,n-1\,.
\end{equation*}
Condition i) is trivially satisfied and 
adapting the proof of step \textbf{(3)} of Theorem~\ref{impli}
we also get Condition iii).
To conclude the proof it is enough to notice that when $\mathbb{R}^n$
is endowed with $\Vert\cdot\Vert_\natural$, condition
$\Vert\Phi\Vert_{com}\leq 1$ is fulfilled 
if $\left\vert\sum_{k=i}^{j-1}\Phi^\perp_k\right\vert\leq 1$
for every $i\leq j-1 \in\{1,\ldots,n\}$, that is nothing else than
$\vert\widetilde{\Phi}_i-\widetilde{\Phi}_{j}\vert\leq 2$.
\end{proof}

We conclude this section proving Lemma \ref{dellefoglie}.

\medskip

\textit{Proof of Lemma~\ref{dellefoglie}.}

\medskip

Denoted by $\mathcal{S}$ a Steiner network connecting the points
of $S$ we repeat the construction of~\cite[Proposition 2.28]{calicipi}
obtaining a suitable labelling of the points of $S$ (and consequently $B_\sigma$). It is possible then to construct a current $T_\sigma=[\Sigma_\sigma,\tau_\sigma,\theta_\sigma]$ 
with boundary $B_\sigma$ such that $\theta_\sigma\in\mathcal{G}$
and $\mathbb{M}_\flat(T_\sigma)=\mathcal{H}^1(\mathcal{S})$: it is enough to define $T_i$ as the $1$--current  supported on the branch of $\mathcal{S}$ connecting $p_i$ with $p_n$ with multiplicity $g_i$ and then build $T_\sigma = \sum_{i=1}^{n-1} T_i$.

Hence by the equivalence between the Steiner problem and Problem~\eqref{minprocurrents}
the current $T_\sigma=[\Sigma_\sigma,\tau_\sigma,\theta_\sigma]$  
is a minimizer for Problem~\eqref{minprocurrents} with $\mathbb{M}= \mathbb{M}_\flat$ and $B = B_\sigma$.
We show that $T_\sigma$ is a minimizer also for Problem~\eqref{sigmaminprocurrents}.
By minimality of $T_\sigma$ it holds 
$\mathbb{M}_{\flat}(T_{\sigma}) \leq \mathbb{M}_{\flat}(T)$
for all $1$--rectifiable current $T$  with coefficients in $\mathbb{Z}^{n-1}$.
Then for all competitors $T$ it holds
\begin{equation*}
\mathbb{M}_\natural(T_\sigma)
=\mathbb{M}_\flat(T_\sigma)\leq \mathbb{M}_\flat(T)
\leq \mathbb{M}_\natural(T)\,,
\end{equation*}
where we used \eqref{minor}. This gives the minimality of $T_\sigma$ for Problem~\eqref{sigmaminprocurrents} 
and concludes the proof.
\qed

\begin{figure}[H]
\begin{tikzpicture}[scale=1.3]
\path[font=\footnotesize]
(1,1) node[above]{$p_2$}
(1,-1) node[below]{$p_3$}
(-1,-1) node[below]{$p_4$}
(-1,1) node[above]{$p_1$}
(0,0) node[right]{$g_1+g_2$}
(0.7,0.5) node[above]{$g_2$}
(0.7,-0.45) node[below]{$g_3$}
(-1.2,-0.45) node[below]{$g_1+g_2+g_3$}
(-0.7,0.5) node[above]{$g_1$};
\fill[black](1,1) circle (1.7pt);    
\fill[black](1,-1) circle (1.7pt);    
\fill[black](-1,1) circle (1.7pt);    
\fill[black](-1,-1) circle (1.7pt);    
\draw[rotate=90]
(1,-1)--(0.42,0)
(1,1)--(0.42,0)
(0.42,0)--(-0.42,0)
(-0.42,0)--(-1,-1)
(-0.42,0)--(-1,1);
\end{tikzpicture}\qquad
\begin{tikzpicture}[scale=1.3]
\path[font=\footnotesize]
(1,1) node[above]{$p_2$}
(1,-1) node[below]{$p_3$}
(-1,-1) node[below]{$p_4$}
(-1,1) node[above]{$p_1$}
(0,0) node[above]{$g_2+g_3$}
(0.6,0.5) node[above]{$g_2$}
(0.6,-0.45) node[below]{$g_3$}
(-1.5,-0.45) node[below]{$g_1+g_2+g_3$}
(-0.6,0.5) node[above]{$g_1$};
\fill[black](1,1) circle (1.7pt);    
\fill[black](1,-1) circle (1.7pt);    
\fill[black](-1,1) circle (1.7pt);    
\fill[black](-1,-1) circle (1.7pt);    
\draw[rotate=0]
(1,-1)--(0.42,0)
(1,1)--(0.42,0)
(0.42,0)--(-0.42,0)
(-0.42,0)--(-1,-1)
(-0.42,0)--(-1,1);
\end{tikzpicture}\qquad
\begin{tikzpicture}[scale=1.3]
\path[font=\footnotesize]
(1,1) node[above]{$p_2$}
(1,-1) node[below]{$p_3$}
(-1,-1) node[below]{$p_4$}
(-1,1) node[above]{$p_1$}
(0,0) node[below]{$g_1+g_3$}
(0.4,0.6) node[above]{$g_2$}
(1.1,-0.45) node[below]{$g_3$}
(-1.5,-0.45) node[below]{$g_1+g_2+g_3$}
(-0.2,1) node[above]{$g_1$};
\fill[black](1,1) circle (1.7pt);    
\fill[black](1,-1) circle (1.7pt);    
\fill[black](-1,1) circle (1.7pt);    
\fill[black](-1,-1) circle (1.7pt);    
\draw[rotate=0]
(-0.42,0)to[out= 120,in=-120, looseness=1](1,1)
(-1,1)to[out= 30,in=90, looseness=1](1.3,1.1)
(1.3,1.1)to[out= -90,in=0, looseness=1](0.42,0)
(1,-1)--(0.42,0)
(0.42,0)--(-0.42,0)
(-0.42,0)--(-1,-1);
\end{tikzpicture}
\caption{With a suitable choice of the labeling of the points 
$p_1,p_2,p_3,p_4$ it is not possible to construct a competitor without loops
that lies in the convex envelop of the points
where a curve has coefficient $g_1+g_3$.}\label{exampleA}
\end{figure}
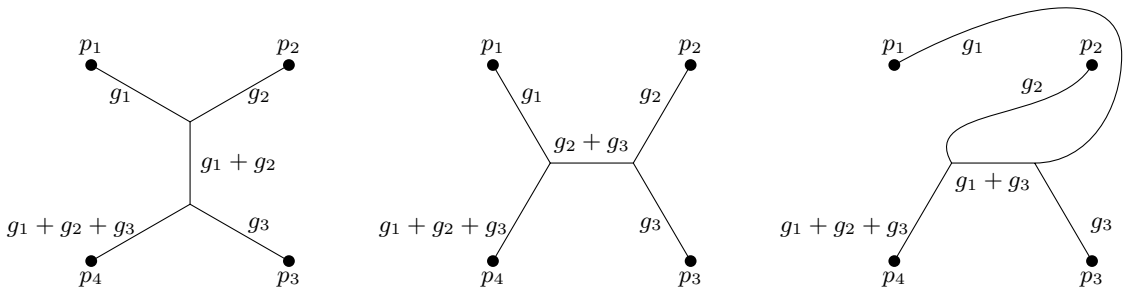

\begin{rem}
Although in general giving explicitly
the permutation $\sigma$ is quite hard,
the choice of a suitable labelling of the  points of $S$
becomes easy when the points 
lie on the boundary of a convex set:
it is enough to label $p_1\ldots,p_n$ in a anticlockwise sense (see Lemma \ref{construction}).
\end{rem}

\subsection{Extension to $\mathbb{R}^n$}

This paper is devoted to compare the 
known notions of calibrations for minimal networks and minimal
partitions in $\mathbb{R}^2$.

\textit{Minimal partitions}

A natural question would be if it is possible to
generalize/modify the mentioned approaches to 
minimal partition problems in higher dimension
with the goal of comparing the related notions of calibrations. 

Paired calibrations are already a tool to validate the minimality
of a certain candidate
 that span a given boundary and
divide the domain (a convex set in $\mathbb{R}^{n+1}$) 
in a fixed number of phases.
At the moment
it is not known if
one can find a suitable group $\mathcal{G}$ and a suitable norm
such that 
$n$--dimensional currents with coefficients in $\mathcal{G}$
represent a partition of $\mathbb{R}^{n+1}$.
Regarding instead the covering spaces approaches 
several attempts to very specific problems 
have been proposed in~\cite{cover,cover2,brakke}.
Despite this remarkable list of examples 
it is still not clear if it is possible to systematically approach 
Plateau's type problems and partition problems within the 
covering space setting.

\textit{Minimal networks}
 
Because of the intrinsic nature of the notion if currents, 
$1$--currents with group coefficients describe networks in any codimension.
Hence this approach is suitable for the Steiner problem in $\mathbb{R}^n$.

To conclude,
Section~\ref{unocorrenti}
is about minimizing $1$--dimensional objects in codimension $n$,
instead Section~\ref{partizioni}
regards the minimization of 
$n$--dimensional objects in codimension $1$.
Clearly $n=2$ is the only case in which the two are comparable.

\section{Convexifications of the problem}\label{convexification}

\begin{dfnz}
We call 
\begin{itemize}
\item   $BV_{constr}(Y, [0,1])$
the space of
functions $u \in BV(Y,[0,1])$ such that for almost every 
$x\in M $ it holds 
$\sum_{p(y) = x} u(y) = 1$
and $u^1(x) = 1$ for every $x\in \R^2 \setminus \Omega$.
\item  $BV_{constr}^\#(Y)$ the space of functions 
in $BV_{constr}(Y, [0,1])$  with a finite number of values
$\alpha_1,\ldots,\alpha_k$.
\end{itemize}
\end{dfnz}

In~\cite{calicipi} we have proven that if $\Phi$ is a calibration for 
$u\in BV_{constr}(Y,\{0,1\})$, then $u$ is a minimizer in the same class,
but actually the following  holds:
\begin{thm}\label{impli}
If $\Phi :Y\to\mathbb{R}^{2}$ is a calibration for $u\in BV_{constr}(Y,\{0,1\})$,
then $u$ is a minimizer among all functions in $BV_{constr}^\#(Y)$.
\end{thm}

For the proof of Theorem~\ref{impli} we need the following:
\begin{lemma}\label{miracle}
Let $\{\eta_i\}_{i=1,\ldots,n}$ and $\{t_i\}_{j=1,\ldots,n}$ such that $\sum_{i=1}^n
\eta_i = 0$ and $|t_i - t_j| \leq 2$ for every $i,j \in 1\ldots,m$. Then
\begin{equation}\label{ooo}
\left|\sum_{i=1}^n t_i \eta_i \right| \leq \sum_{i=1}^n |\eta_i|\,.
\end{equation}
\end{lemma}
\begin{proof}
Notice that 
\begin{eqnarray*}
\left|\sum_{i=1}^n t_i \eta_i \right| &=& \left|\sum_{\eta_i > 0} t_i \eta_i +
\sum_{\eta_i < 0} t_i \eta_i\right| \leq \left|\sum_{\eta_i > 0} \max(t_i) \eta_i +
\sum_{\eta_i < 0} \min(t_i) \eta_i\right| \\
&=& \left|\sum_{\eta_i > 0} \max(t_i) \eta_i - \sum_{\eta_i > 0} \min(t_i)
\eta_i\right| \leq 2\left|\sum_{\eta_i > 0} \eta_i\right| = \sum_{i=1}^n |\eta_i|\,.
\end{eqnarray*}
\end{proof}

\begin{rem}\label{divteo} 
 We also note that given $u,w \in BV_{constr}(Y,[0,1])$ and
$\Phi : Y\to \R^2$ an approximately regular divergence free
vector field it holds
\begin{equation}\label{booh}
\int_{Y} \Phi \cdot Du = \int_{Y} \Phi \cdot Dw\,.
\end{equation}    
This result can be proved adapting~\cite[Proposition 4.3]{calicipi}.
\end{rem}

\textit{Proof of Theorem~\ref{impli}.}

Consider $u\in BV_{constr}(Y,\{0,1\})$, 
$\Phi :Y\to\mathbb{R}^{2}$ a calibration for $u$
and $w$ a competitor in $BV_{constr}^\#(Y)$. 
Combining Remark~\ref{divteo} with Conditions (\textbf{1}) and (\textbf{3})
of Definition~\ref{caliconvering} we have
\begin{equation}\label{b}
|Du|(Y) = \int_{Y} \Phi \cdot Du= \int_{Y} \Phi \cdot Dw\,.
\end{equation}
Moreover by the representation formula for $Dw$ in the space $Y$  we get
\begin{equation*}
 \int_{Y} \Phi \cdot Dw = \sum_{j=1}^m \int_{\R^2} \Phi^j \cdot Dw^j\,, 
\end{equation*}
where 
without loss of generality we have supposed that $p(J_w) \cap \Sigma=\emptyset$ (see Remark \ref{independent}).

Calling 
$\eta_j (x) = (w^j)^+(x) - (w^j)^+(x)$
(we refer to Remark~\ref{presc} for the definition of $w^j$), 
we notice that, as $w \in BV_{constr}^\#(Y)$, for almost every $x \in \R^2$
\begin{equation*}
\sum_{j=1}^m \eta_j (x) = 0 \,.
\end{equation*}
Then
\begin{align*}
\sum_{j=1}^m \int_{\R^2} \Phi^j \cdot Dw^j  
&= \sum_{j=1}^m \int_{J_{w^j}} \eta_j \Phi^j \cdot \nu \, d\Ha^{1}
\nonumber\\
&= \int_{p(J_w)}\sum_{j=1}^m  \eta_j (\Phi^j \cdot \nu) \chi_{J_w^j} \,
d\Ha^{1} \\
&\leq  \int_{p(J_w)}\Big|\sum_{j=1}^m \Phi^j \eta_j \chi_{J_w^j} \Big|\,
d\Ha^{1} \,.
\end{align*}
Applying Lemma~\ref{miracle} one obtains 
\begin{equation}\label{c}
\int_{Y} \Phi \cdot Dw \leq  \int_{p(J_w)} \sum_{j=1}^m |\eta_j|\chi_{J_w^j}\,
d\Ha^1 = |Dw|(Y)\,.
\end{equation}
Hence combining \eqref{b} with \eqref{c} we conclude that
\begin{equation*}
|Du|(Y) = \int_{Y} \Phi \cdot Du
= \int_{Y} \Phi \cdot Dw\leq |Dw|(Y)\,.
\end{equation*}
\qed

\begin{rem}
The previous theorem is sharp, in the sense that one cannot replace
$BV_{constr}^\#(Y)$  by $BV_{constr}(Y, [0,1])$.
Indeed consider $S=\{p_1,p_2,p_3\}$ with $p_i$
vertices of an equilateral triangle.
Although the minimizer $u\in BV_{constr}(Y, \{0,1\})$
is calibrated (for the result in our setting see~\cite[Example 3.8]{calicipi}),
there exists a function in  $BV_{constr}(Y, [0,1])\setminus BV_{constr}^\#(Y)$ 
whose total variation is strictly less than the total variation of $u$, 
as it shown in~\cite[Proposition 5.1]{chambollecremerspock}.
\end{rem}


We define now the convexification of 
$\vert Du\vert$ with $u\in BV_{constr}(Y,\{0,1\})$ naturally associated to the notion of calibration for covering spaces. 
It was introduced 
by Chambolle, Cremers and Pock in~\cite{chambollecremerspock} in the context of minimal partitions.

\begin{dfnz}[Local convex envelope]
Let $u\in BV_{constr}(Y, [0,1])$.
We consider the functional 
$G$ given by
\begin{equation*}
G(u):=\int_{Y}\Psi(Du)\,,
\end{equation*}
where  
\begin{equation*}
\Psi(q) = \sup_{p\in K} \sum_{j=1}^n p^j \cdot q^j
\end{equation*}
and 
\begin{equation*}
K = \{p \in Y : \|p^i - p^j\| \leq 2, \mbox{ for every } i,j = 1,\ldots, n\} \,.
\end{equation*}
In analogy with~\cite{chambollecremerspock}  we call $G$ 
local convex envelope.
\end{dfnz}

The local convex envelope is the tightest convexification with an integral form, indeed:

\begin{prop}{(\cite{chambollecremerspock})}\label{best}
The  local convex envelope $G$ is the larger 
convex integral functional   of the form 
$H(v)=\int_{Y} \Psi(x, Dv)$
with $v\in BV_{constr}(Y, [0,1])$ and $\Psi(x,\cdot)$ non negative,
even and convex such that 
 \begin{equation*}
 H(v)=\vert Dv\vert(Y) \quad \text{for} \;v\in BV_{constr}(Y,\{0,1\})\,.
\end{equation*}
\end{prop}

As a consequence of Theorem~\ref{impli} we are able to prove:

\begin{prop}\label{equal}
It holds
 \begin{equation*}
G(v)=\vert Dv\vert(Y) \quad \text{for} \;v\in BV_{constr}^\#(Y)\,.
\end{equation*}
\end{prop}
\begin{proof}
The inequality $G(v) \geq \vert Dv\vert(Y)$ is consequence of 
Proposition~\ref{best} choosing $\Psi(x,p) = |p|$. For the other inequality it is just enough 
to notice that given $v \in BV^\#_{constr}(Y)$, from the proof of Theorem~\ref{impli} we obtain that
\begin{equation*}
\int_{Y} p \cdot Dv \leq \vert Dv\vert(Y)
\end{equation*}
for every $p\in K$. Therefore taking the supremum 
on both sides and using that $|Dv|(Y) < +\infty$ we conclude that
\begin{equation*}
\int_{Y}  \sup_{p\in K}  p \cdot Dv \leq \vert Dv\vert(Y) \, .
\end{equation*}
\end{proof}
\begin{rem}
Proposition \ref{equal} shows that even if the local convex envelope is the best integral convexification of the problem, it "outperforms" the total variation only when evaluated on functions $u\in BV_{constr}(Y,[0,1]) \setminus BV_{constr}^\#(Y)$.
\end{rem}

\section{An example of nonexistence of calibrations}\label{nonexistence}

Finding a calibration for a candidate minimizer is not an easy task.
We wanted to understand at least
whether there exists a calibration when $S$ is composed 
of points lying at the vertices of a regular polygon.
We have a positive answer only in the case of a triangle and of a 
square~\cite{calicipi, annalisaandrea}.
As a byproduct of Theorem~\ref{impli} we are now able to 
negatively answer in the case of a regular pentagon.

\begin{ex}[Five vertices of a regular pentagon]\label{fivepoints}
Given $S=\{p_1,\ldots,p_5\}$ with $p_i$  the five vertices of a regular pentagon,
the minimizer of the Steiner problem is well known.
Following the canonical construction presented in~\cite[Proposition 2.28]{calicipi}
it is not difficult to construct the ``associated" 
function  $u\in BV_{constr}(Y,\{0,1\})$
here represented in Figure~\ref{pentaoncovering}.
By explicit computations one gets $\frac{1}{2}\vert Du\vert (Y)=4.7653$.
Consider now the function $w\in BV_{constr}(Y, \{0,\frac12,1\})$
exhibited in Figure~\ref{pentaoncovering}.
It holds
$\frac{1}{2}\vert Dw\vert (Y) \approx 4.5677$.
Theorem~\ref{impli} tells us that if a calibration for the minimizer 
$u$ exists, then $u$ has to be a minimizer also in the larger space
$BV_{constr}^\#(Y)$, but $\vert Dw\vert (Y) <\vert Du\vert(Y)$, 
hence  a calibration
for $u$ does not exists.

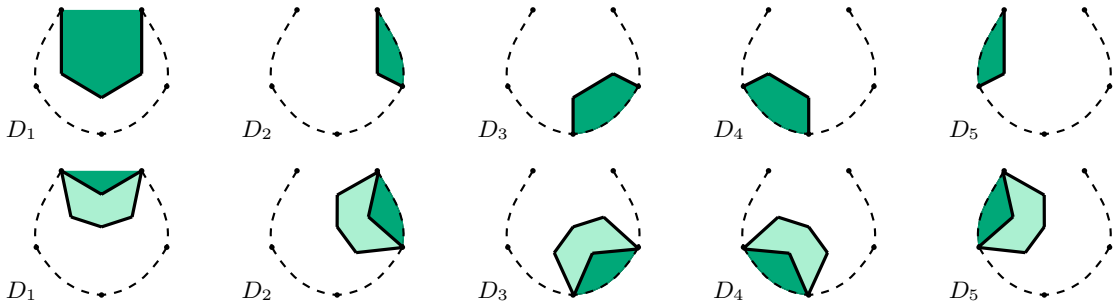
\begin{figure}[H]
\begin{tikzpicture}[scale=0.6]
\filldraw[fill=green(munsell)]
(0.88,1.24)--(0.88,-0.17)--
(0.88,-0.17)--(0,-0.7)--
(0,-0.7)--(-0.88,-0.17)--
(-0.88,1.24);
\draw[color=black, dashed, thick]
(-0.88,1.24)to[out= -120,in=100, looseness=1] (-1.43,-0.45) 
(0.88,1.24)to[out= -60,in=80, looseness=1] (1.43,-0.45)  
(-1.43,-0.45) to[out= -60,in=175, looseness=1] (0,-1.505)
(1.43,-0.45) to[out= -120,in=5, looseness=1] (0,-1.505);
\path[font=\footnotesize]
(-1.75,-1.85) node[above]{$D_1$};
\fill[black](-0.88,1.24) circle (1.7pt);    
 \fill[black](0.88,1.24) circle (1.7pt);
\fill[black](1.43,-0.45) circle (1.7pt);
\fill[black](-1.43,-0.45) circle (1.7pt);
\fill[black](0,-1.505) circle (1.7pt);
\draw[very thick]
(0.88,1.24)--(0.88,-0.17)
(-0.88,1.24)--(-0.88,-0.17)
(-0.88,-0.17)--(0,-0.7)
(0.88,-0.17)--(0,-0.7);
\end{tikzpicture}\qquad
\begin{tikzpicture}[scale=0.6]
\filldraw[fill=green(munsell)]
(0.88,1.24)to[out= -60,in=80, looseness=1] (1.43,-0.45)    --
 (1.43,-0.45)--(0.88,-0.17)--
 (0.88,-0.17)--(0.88,1.24);
 \draw[green(munsell)]
(0.88,1.24)to[out= -60,in=80, looseness=1] (1.43,-0.45)  ;
\draw[color=black, dashed, thick]
(-0.88,1.24)to[out= -120,in=100, looseness=1] (-1.43,-0.45) 
(0.88,1.24)to[out= -60,in=80, looseness=1] (1.43,-0.45)  
(-1.43,-0.45) to[out= -60,in=175, looseness=1] (0,-1.505)
(1.43,-0.45) to[out= -120,in=5, looseness=1] (0,-1.505);
\path[font=\footnotesize]
(-1.75,-1.85) node[above]{$D_2$};
\fill[black](-0.88,1.24) circle (1.7pt);    
 \fill[black](0.88,1.24) circle (1.7pt);
\fill[black](1.43,-0.45) circle (1.7pt);
\fill[black](-1.43,-0.45) circle (1.7pt);
\fill[black](0,-1.505) circle (1.7pt);
\draw[very thick]
(0.88,1.24)--(0.88,-0.17)
(0.88,-0.17)--(1.43,-0.45);
\end{tikzpicture}\qquad
\begin{tikzpicture}[scale=0.6]
\filldraw[fill=green(munsell)]
(1.43,-0.45) to[out= -120,in=5, looseness=1] (0,-1.505) --
(0,-1.505)--(0,-0.7)--
(0,-0.7)--(0.88,-0.17)--
(0.88,-0.17)--(1.43,-0.45);
 \draw[green(munsell)]
(1.43,-0.45) to[out= -120,in=5, looseness=1] (0,-1.505) ;
\draw[color=black, dashed, thick]
(-0.88,1.24)to[out= -120,in=100, looseness=1] (-1.43,-0.45) 
(0.88,1.24)to[out= -60,in=80, looseness=1] (1.43,-0.45)  
(-1.43,-0.45) to[out= -60,in=175, looseness=1] (0,-1.505)
(1.43,-0.45) to[out= -120,in=5, looseness=1] (0,-1.505);
\path[font=\footnotesize]
(-1.75,-1.85) node[above]{$D_3$};
\fill[black](-0.88,1.24) circle (1.7pt);    
 \fill[black](0.88,1.24) circle (1.7pt);
\fill[black](1.43,-0.45) circle (1.7pt);
\fill[black](-1.43,-0.45) circle (1.7pt);
\fill[black](0,-1.505) circle (1.7pt);
\draw[very thick]
(1.43,-0.45)--(0.88,-0.17)
(0.88,-0.17)--(0,-0.7)
(0,-0.7)--(0,-1.505);
\end{tikzpicture}\qquad
\begin{tikzpicture}[scale=0.6]
\filldraw[fill=green(munsell)]
(-1.43,-0.45) to[out= -60,in=175, looseness=1] (0,-1.505) --
(0,-1.505)--(0,-0.7)--
(0,-0.7)--(-0.88,-0.17)--
(-0.88,-0.17)--(-1.43,-0.45);
 \draw[green(munsell)]
(-1.43,-0.45) to[out= -60,in=175, looseness=1] (0,-1.505) ;
\draw[color=black, dashed, thick]
(-0.88,1.24)to[out= -120,in=100, looseness=1] (-1.43,-0.45) 
(0.88,1.24)to[out= -60,in=80, looseness=1] (1.43,-0.45)  
(-1.43,-0.45) to[out= -60,in=175, looseness=1] (0,-1.505)
(1.43,-0.45) to[out= -120,in=5, looseness=1] (0,-1.505);
\path[font=\footnotesize]
(-1.75,-1.85) node[above]{$D_4$};
\fill[black](-0.88,1.24) circle (1.7pt);    
 \fill[black](0.88,1.24) circle (1.7pt);
\fill[black](1.43,-0.45) circle (1.7pt);
\fill[black](-1.43,-0.45) circle (1.7pt);
\fill[black](0,-1.505) circle (1.7pt);
\draw[very thick]
(-1.43,-0.45)--(-0.88,-0.17)
(-0.88,-0.17)--(0,-0.7)
(0,-0.7)--(0,-1.505);
\end{tikzpicture}\qquad
\begin{tikzpicture}[scale=0.6]
\filldraw[fill=green(munsell)]
(-0.88,1.24)to[out= -120,in=100, looseness=1] (-1.43,-0.45) --
 (-1.43,-0.45)--(-0.88,-0.17)--
 (-0.88,-0.17)--(-0.88,1.24);
 \draw[green(munsell)]
 (-0.88,1.24)to[out= -120,in=100, looseness=1] (-1.43,-0.45) ;
\draw[color=black, dashed, thick]
(-0.88,1.24)to[out= -120,in=100, looseness=1] (-1.43,-0.45) 
(0.88,1.24)to[out= -60,in=80, looseness=1] (1.43,-0.45)  
(-1.43,-0.45) to[out= -60,in=175, looseness=1] (0,-1.505)
(1.43,-0.45) to[out= -120,in=5, looseness=1] (0,-1.505);
\path[font=\footnotesize]
(-1.75,-1.85) node[above]{$D_5$};
\fill[black](-0.88,1.24) circle (1.7pt);    
 \fill[black](0.88,1.24) circle (1.7pt);
\fill[black](1.43,-0.45) circle (1.7pt);
\fill[black](-1.43,-0.45) circle (1.7pt);
\fill[black](0,-1.505) circle (1.7pt);
\draw[very thick]
(-0.88,1.24)--(-0.88,-0.17)
(-0.88,-0.17)--(-1.43,-0.45);
\end{tikzpicture}

\medskip

\begin{tikzpicture}[scale=0.6]
\filldraw[fill=green(munsell)]
(0.88,1.24)--(0,0.72)--
(-0.88,1.24);
\filldraw[fill=magicmint]
(0,0.72)--(-0.88,1.24)--
(-0.88,1.24)--(-0.67,0.22)--
(-0.67,0.22)--(0,0)--
(0,0)--(0.67,0.22)--
(0.67,0.22)--(0.88,1.24)
(0.88,1.24)--(0,0.72);
\draw[very thick, black]
(-0.88,1.24)--(-0.67,0.22)
(-0.67,0.22)--(0,0)
(0,0)--(0.67,0.22)
(0.67,0.22)--(0.88,1.24)
(0.88,1.24)--(0,0.72)
(-0.88,1.24)--(0,0.72);
\draw[color=black, dashed, thick]
(-0.88,1.24)to[out= -120,in=100, looseness=1] (-1.43,-0.45) 
(0.88,1.24)to[out= -60,in=80, looseness=1] (1.43,-0.45)  
(-1.43,-0.45) to[out= -60,in=175, looseness=1] (0,-1.505)
(1.43,-0.45) to[out= -120,in=5, looseness=1] (0,-1.505);
\path[font=\footnotesize]
(-1.75,-1.85) node[above]{$D_1$};
\fill[black](-0.88,1.24) circle (1.7pt);    
 \fill[black](0.88,1.24) circle (1.7pt);
\fill[black](1.43,-0.45) circle (1.7pt);
\fill[black](-1.43,-0.45) circle (1.7pt);
\fill[black](0,-1.505) circle (1.7pt);
\end{tikzpicture}\qquad
\begin{tikzpicture}[scale=0.6]
\filldraw[fill=green(munsell)]
(0.88,1.24)to[out= -60,in=80, looseness=1] (1.43,-0.45)    --
 (1.43,-0.45)--(0,-0)--
 (0,0)--(0.88,1.24);
\filldraw[fill=magicmint, rotate=-72]
(0,0.72)--(-0.88,1.24)--
(-0.88,1.24)--(-0.67,0.22)--
(-0.67,0.22)--(0,0)--
(0,0)--(0.67,0.22)--
(0.67,0.22)--(0.88,1.24)
(0.88,1.24)--(0,0.72);
 \draw[green(munsell)]
(0.88,1.24)to[out= -60,in=80, looseness=1] (1.43,-0.45)  ;
\draw[color=black, dashed, thick]
(-0.88,1.24)to[out= -120,in=100, looseness=1] (-1.43,-0.45) 
(0.88,1.24)to[out= -60,in=80, looseness=1] (1.43,-0.45)  
(-1.43,-0.45) to[out= -60,in=175, looseness=1] (0,-1.505)
(1.43,-0.45) to[out= -120,in=5, looseness=1] (0,-1.505);
\path[font=\footnotesize]
(-1.75,-1.85) node[above]{$D_2$};
\fill[black](-0.88,1.24) circle (1.7pt);    
 \fill[black](0.88,1.24) circle (1.7pt);
\fill[black](1.43,-0.45) circle (1.7pt);
\fill[black](-1.43,-0.45) circle (1.7pt);
\fill[black](0,-1.505) circle (1.7pt);
\draw[very thick, black, rotate=-72]
(-0.88,1.24)--(-0.67,0.22)
(-0.67,0.22)--(0,0)
(0,0)--(0.67,0.22)
(0.67,0.22)--(0.88,1.24)
(0.88,1.24)--(0,0.72)
(-0.88,1.24)--(0,0.72);
\end{tikzpicture}\qquad
\begin{tikzpicture}[scale=0.6]
\filldraw[fill=green(munsell)]
(1.43,-0.45) to[out= -120,in=5, looseness=1] (0,-1.505) --
(0,-1.505)--(0,0)--
(0,0)--(1.43,-0.45);
\filldraw[fill=magicmint, rotate=-144]
(0,0.72)--(-0.88,1.24)--
(-0.88,1.24)--(-0.67,0.22)--
(-0.67,0.22)--(0,0)--
(0,0)--(0.67,0.22)--
(0.67,0.22)--(0.88,1.24)
(0.88,1.24)--(0,0.72);
 \draw[green(munsell)]
(1.43,-0.45) to[out= -120,in=5, looseness=1] (0,-1.505) ;
\draw[color=black, dashed, thick]
(-0.88,1.24)to[out= -120,in=100, looseness=1] (-1.43,-0.45) 
(0.88,1.24)to[out= -60,in=80, looseness=1] (1.43,-0.45)  
(-1.43,-0.45) to[out= -60,in=175, looseness=1] (0,-1.505)
(1.43,-0.45) to[out= -120,in=5, looseness=1] (0,-1.505);
\path[font=\footnotesize]
(-1.75,-1.85) node[above]{$D_3$};
\fill[black](-0.88,1.24) circle (1.7pt);    
 \fill[black](0.88,1.24) circle (1.7pt);
\fill[black](1.43,-0.45) circle (1.7pt);
\fill[black](-1.43,-0.45) circle (1.7pt);
\fill[black](0,-1.505) circle (1.7pt);
\draw[very thick, black, rotate=-144]
(-0.88,1.24)--(-0.67,0.22)
(-0.67,0.22)--(0,0)
(0,0)--(0.67,0.22)
(0.67,0.22)--(0.88,1.24)
(0.88,1.24)--(0,0.72)
(-0.88,1.24)--(0,0.72);
\end{tikzpicture}\qquad
\begin{tikzpicture}[scale=0.6]
\filldraw[fill=green(munsell)]
(-1.43,-0.45) to[out= -60,in=175, looseness=1] (0,-1.505) --
(0,-1.505)--(0,0)--
(0,0)--(-1.43,-0.45);
\filldraw[fill=magicmint, rotate=144]
(0,0.72)--(-0.88,1.24)--
(-0.88,1.24)--(-0.67,0.22)--
(-0.67,0.22)--(0,0)--
(0,0)--(0.67,0.22)--
(0.67,0.22)--(0.88,1.24)
(0.88,1.24)--(0,0.72);
 \draw[green(munsell)]
(-1.43,-0.45) to[out= -60,in=175, looseness=1] (0,-1.505) ;
\draw[color=black, dashed, thick]
(-0.88,1.24)to[out= -120,in=100, looseness=1] (-1.43,-0.45) 
(0.88,1.24)to[out= -60,in=80, looseness=1] (1.43,-0.45)  
(-1.43,-0.45) to[out= -60,in=175, looseness=1] (0,-1.505)
(1.43,-0.45) to[out= -120,in=5, looseness=1] (0,-1.505);
\path[font=\footnotesize]
(-1.75,-1.85) node[above]{$D_4$};
\fill[black](-0.88,1.24) circle (1.7pt);    
 \fill[black](0.88,1.24) circle (1.7pt);
\fill[black](1.43,-0.45) circle (1.7pt);
\fill[black](-1.43,-0.45) circle (1.7pt);
\fill[black](0,-1.505) circle (1.7pt);
\draw[very thick, black, rotate=144]
(-0.88,1.24)--(-0.67,0.22)
(-0.67,0.22)--(0,0)
(0,0)--(0.67,0.22)
(0.67,0.22)--(0.88,1.24)
(0.88,1.24)--(0,0.72)
(-0.88,1.24)--(0,0.72);
\end{tikzpicture}\qquad
\begin{tikzpicture}[scale=0.6]
\filldraw[fill=green(munsell)]
(-0.88,1.24)to[out= -120,in=100, looseness=1] (-1.43,-0.45) --
 (-1.43,-0.45)--(0,0)--
 (0,0)--(-0.88,1.24);
 \filldraw[fill=magicmint, rotate=72]
(0,0.72)--(-0.88,1.24)--
(-0.88,1.24)--(-0.67,0.22)--
(-0.67,0.22)--(0,0)--
(0,0)--(0.67,0.22)--
(0.67,0.22)--(0.88,1.24)
(0.88,1.24)--(0,0.72);
 \draw[green(munsell)]
 (-0.88,1.24)to[out= -120,in=100, looseness=1] (-1.43,-0.45) ;
\draw[color=black, dashed, thick]
(-0.88,1.24)to[out= -120,in=100, looseness=1] (-1.43,-0.45) 
(0.88,1.24)to[out= -60,in=80, looseness=1] (1.43,-0.45)  
(-1.43,-0.45) to[out= -60,in=175, looseness=1] (0,-1.505)
(1.43,-0.45) to[out= -120,in=5, looseness=1] (0,-1.505);
\path[font=\footnotesize]
(-1.75,-1.85) node[above]{$D_5$};
\fill[black](-0.88,1.24) circle (1.7pt);    
\fill[black](0.88,1.24) circle (1.7pt);
\fill[black](1.43,-0.45) circle (1.7pt);
\fill[black](-1.43,-0.45) circle (1.7pt);
\fill[black](0,-1.505) circle (1.7pt);
\draw[very thick, black, rotate=72]
(-0.88,1.24)--(-0.67,0.22)
(-0.67,0.22)--(0,0)
(0,0)--(0.67,0.22)
(0.67,0.22)--(0.88,1.24)
(0.88,1.24)--(0,0.72)
(-0.88,1.24)--(0,0.72);
\end{tikzpicture}
\caption{Up: The functions $u\in BV_{constr}(Y,\{0,1\})$ minimizer of Problem~\eqref{minpro2}
Down: A function $w$ in $BV_{constr}(Y,\{0,1/2,1\})$ with 
$\vert Dw\vert(Y)<\vert Du\vert(Y)$.
White corresponds to the value $0$, light green to $1/2$ and 
dark green to $1$.}\label{pentaoncovering}
\end{figure}

Non--existence of calibrations for minimal currents when 
$S=\{p_1,\ldots,p_5\}$ with $p_i$  the five vertices of a regular pentagon
was already highlighted 
in~\cite[Example 4.2]{bonafini}.
As we have shown that Definition~\ref{calicurrents}
is stronger than Definition~\ref{caliconvering} it
was necessary to  ``translate" the example in our setting
to conclude that a calibration does not exists for $u$ minimizer of Problem~\eqref{minpro2}.
\end{ex}

\begin{rem}
We refer also to~\cite[Example 4.6]{annalisaandrea}
where an example of nonexistence of calibrations is provided.
However in that case the ambient space is not $\mathbb{R}^2$
endowed with the standard Euclidean metric.
\end{rem}

\subsection{Remarks on calibrations in families}

Example~\ref{fivepoints} underlines an big issue
of the theory of calibrations.
\emph{Calibrations in families} (see~\cite[Section 4]{calicipi}) can
avoid the problem.
Indeed in~\cite[Example 4.9]{calicipi} we are able to find the minimal
Steiner network
for the five vertices of a regular pentagon via a calibration argument.
We explain briefly here the strategy we used and we validate it
with some remarks.

\begin{itemize}
\item First we divide the sets of $\mathscr{P}_{constr}(Y)$  in families.
The competitors that belong to the same class
share a property related to the projection of their essential boundary
onto the base set $M$.
In particular we define a family as
\begin{equation*}\label{famiglia}
\mathcal{F}(\mathcal{J}) := \{E \in \mathscr{P}_{constr}(Y_\Sigma):
 \Ha^1(E^{i,j}) \neq 0 \mbox{ for every } (i,j) \in \mathcal{J}\}.
\end{equation*}
where $\mathcal{J} \subset \{1,\ldots,m\}\times \{1,\ldots,m\}$  and
$E^{i,j}:= \partial^\ast E^i\cap \partial^\ast E^j$.
The union of the families has to cover $\mathscr{P}_{constr}(Y)$.
\item We consider a suitable notion of calibrations for $E$ in
$\mathcal{F}(\mathcal{J})$:
Condition \textbf{(2)} can be weaken as

$|\Phi^i(x) - \Phi^j(x)| \leq 2$ for every $i,j = 1,\ldots m$
such that $(i,j) \in \mathcal{J}$ and for every $x\in D$.
\item We calibrate the candidate minimizer in each family.
\item We compare the perimeter of each calibrated minimizer
to find the explicit global
minimizers of Problem~\eqref{minpro2}.
\end{itemize}

\textit{How to divide the competitors in families}

We consider as competitors only the sets in
$\mathscr{P}_{constr}(Y)$ whose projection onto $M$
is a network without loops.
Since it is known that the minimizers are tree--like, 
the previous choice is not restrictive.

\smallskip

Suppose that $S$ consists of $n$ points
located on the boundary of a convex set $\Omega$.
Then Problem~\eqref{minpro2} is equivalent to a minimal partition problem and
$E\in\mathscr{P}_{constr}(Y)$ induces a partition
$\{A_1,\ldots,A_n\}$ of $\Omega$. 
We classify the sets in $\mathscr{P}_{constr}(Y)$ simply prescribing
which phases  ``touch" each other (see~\cite[Lemma 4.8]{calicipi}).
The division in families depends on the topology of the complementary of
the network.

\smallskip

Let us now pass to the general case of any configuration of $n$ points of
$S$.
The minimal Steiner networks are composed of
at most $m=2n-3$ segments.
Each segment of a minimizer $\mathcal{S}$ coincides with
$\overline{p(E^{i,j})}$
for $i\neq j\in \{1,\ldots,n\}$.
Different segments of $\mathcal{S}$ are associated with different $E^{i,j}$.
We take $\mathcal{I}$ composed of  $2n-3$ different couples
of indices  $(i,j)\in \{1,\ldots,n\}\times \{1,\ldots,n\}$.
The cover of $\mathscr{P}_{constr}(Y)$ is given by considering
all possible $\mathcal{I}$  satisfying the above property.


\medskip

\textit{Existence of calibrations in families}

The just proposed division in families is the finest possible one
and it classifies the competitors relying on their topological type.
Note that the length is a convex function of the location of the junctions.
As a consequence each stationary network
is the unique minimizer in its topological type
(see 
for instance~\cite[Corollary  4.3]{morgancluster} and~\cite{choe} where more general situations
are treated)
and therefore a calibration in such a
family always exists.

\medskip

%
%

\textit{Export the idea of calibrations in families to currents}

Once identified the families for sets in $\mathscr{P}_{constr}(Y)$
it is possible to produce families
for Problems~\eqref{minprocurrents} and~\eqref{sigmaminprocurrents}.
Take a competitor in each $\mathcal{F}(\mathcal{J})$.
When the points of $S$ lie on the boundary of a convex set
it is sufficient to apply Lemma~\ref{construction}
(reminding that $E^i=A_{n+1-i}$) to construct a current $T$.
Then one can identify the coefficients of $T$.
Hence in this case
the classification in families will rely on which subsums of $g_i$ are
present in the competitors.

To deal with the case of general configurations of points of $S$,
we have to generalize Lemma~\ref{construction}.
In the construction
we set $T_i=[\partial^\ast E^i,\tau_i,e_i]$ where $\tau_i$ are the tangent vectors to $\partial^\ast E^i$ and the multiplicities
are chosen
in such a way that $e_i-e_{i-1}=\tilde{g}_i$ with $\tilde{g}_i$ linearly
independent
vectors of $\mathbb{R}^{n-1}$.
Again we set $T=\sum T_i$.  Now $\partial T$ is the sum of
$\tilde{g}_j\delta_{p_i}$
where $j$ can also be different from $i$.
We obtain a current with the desired boundary simply substituting
$\tilde{g}_j$ with $g_i$ in order to satisfy
$\tilde{g}_j\delta_{p_i}=g_i\delta_{p_i}$.

\section*{Appendix: Regularity of the calibration}

\begin{prop}[Constancy theorem for currents with coefficients in $\R^n$]\label{constancy}
Let $T$ be a normal $2$-current in $\R^2$ with coefficients in $\R^n$. Then there exists $u\in BV(\R^2,\R^n)$ such that for every $\omega \in C_c^\infty(\R^2,\R^n)$
\begin{equation}
T(\omega) =  \int_{\R^2} \langle \omega, u\rangle \, d\mathscr{L}^2\,.
\end{equation}
\end{prop}
\begin{proof}
Notice firstly that the space of $2$--forms with values in $\R^n$ can be identified by Hodge duality to the space $C_c^\infty(\R^2,\R^n)$.
As $T$ is a normal current, by Riesz theorem there exists $\sigma: \R^2 \rightarrow \R^n$ and $\mu_T$ a finite measure in $\R^2$ such that
\begin{equation}\label{const1}
T(\omega)  = \int_{\R^2} \langle \omega , \sigma\rangle \, d\mu_T = \sum_{i=1}^n\int_{\R^2} \omega_i \sigma_i \, d\mu_T\,.
\end{equation}
Defining $T_i : C_c^\infty(\R^2,\R) \rightarrow \R$ as
\begin{equation*}
T_i(f) = \int_{\R^2} f\sigma_i \, d\mu_T
\end{equation*}
we know that $T_i$ is a $2$-normal current with coefficients in $\R$. Therefore we can apply the standard constancy theorem (see for instance~\cite[\S 3.2,~Theorem 3]{CCC}) and find $u_i \in BV(\R^2)$ such that
\begin{equation}\label{const2}
T_i(f) = \int_{\R^2} f u_i \, d\mathscr{L}^2
\end{equation}
for every $i=1,\ldots,n$.
Hence combining \eqref{const1} and \eqref{const2} we conclude.
\end{proof}

We recall the definition of approximately regular vector fields both on $\R^n$ and on the covering space $Y$ (\cite{mumford, calicipi}).

\begin{dfnz}[Approximately regular vector fields on $\R^n$]
Given $A\subset \R^{n}$, a Borel vector field $\Phi: A \rightarrow \R^{n}$ is approximately regular
if it is bounded and for every Lipschitz hypersurface $M$ in $\R^{n}$, $\Phi$ admits traces on $M$ on the two sides of $M$ (denoted by $\Phi^+$ and $\Phi^-$) and 
\begin{equation}\label{app}
\Phi^+(x) \cdot \nu_M(x) = \Phi^-(x) \cdot \nu_M(x) = \Phi(x) \cdot \nu_M(x),
\end{equation}
for $\mathcal{H}^{n-1}$--a.e. $x \in M\cap A$.
\end{dfnz}  

\begin{dfnz}[Approximately regular vector fields on the covering $Y$]\label{approxi}
Given $\Phi: Y\rightarrow \R^2$, we say that it is \emph{approximately regular} in $Y$ if
$\Phi^j$
 is \emph{approximately regular} for every $j=1,\ldots,m$.  
\end{dfnz}

\begin{thm}\label{approximately}
Suppose that $\Phi:\mathbb{R}^2 \rightarrow M^{n\times 2}(\mathbb{R})$ is a matrix valued vector field such that its rows are approximately regular vector fields.
Given $T=[\Sigma, \tau, \theta]$ a
$1$--rectifiable current with coefficients in $\mathbb{Z}^n$, assume that $\Phi$ satisfies condition (i), (ii) and (iii) of Definition \ref{calicurrents}.
Then $T$ is mass minimizing among all rectifiable $1$--currents with coefficients
in $\mathbb{Z}^n$  in its homology class.
\end{thm}
\begin{proof}
Set $\Omega' \subset\subset \Omega \subset \R^2$ open, bounded, smooth sets such that they contain the convex envelope of $S$.
Given the candidate minimizer $T$ we take a competitor 
$\widetilde T = [\widetilde{\Sigma}, \widetilde \tau, \widetilde \theta]$: a rectifiable $1$--current in $\mathbb{R}^2$ with coefficients in $\mathbb{Z}^n$
 such that 
$\partial(T-\widetilde T)= 0$. Notice that we can suppose that $\Sigma, \widetilde \Sigma \subset \Omega'$.
There exists $U$ a normal $2$--current in $\mathbb{R}^2$ with coefficients in $\mathbb{Z}^n$
such that $T - \widetilde T = \partial U$. 
By Proposition \ref{constancy} there exists $u\in BV(\R^2, \R^n)$ such that for every $\omega\in C_c^\infty(\R^2, \R^n)$
\begin{displaymath}
U(\omega) = \int_{\R^2} \langle \omega, u\rangle \, d\mathscr{L}^{2}\,.
\end{displaymath}
Notice that for every $\phi \in C_c^\infty(\R^2,M^{n\times 2})$ supported in $\R^2 \setminus \Omega'$ we have
\begin{equation*}
0 = T(\phi) - \widetilde T(\phi) = -U(d\phi) = - \int_{\R^2}\langle u,d\phi\rangle \, d\mathscr{L}^2 = \sum_{i=1}^n \int_{\R^2} u_i \div \phi_i^\perp\, d\mathscr{L}^2\,. 
\end{equation*}

Taking the supremum on $\phi \in C_c^\infty(\R^2, M^{n\times 2})$ compactly supported in $\R^2 \setminus \Omega'$ such that $\|\omega\|_\infty \leq 1$
we infer that $|Du|(\R^2 \setminus \Omega') = 0$ and therefore there exists a vector $c\in \R^n$ such that $u(x) = c$ in $\R^2 \setminus \Omega'$ almost everywhere. Define then
\begin{equation}
U_0(\omega) = \int_{\R^2} \langle \omega,u^c\rangle\, d\mathscr{L}^2
\end{equation}
where $u^c(x) = u(x) - c$. It is easy to check that $U_0(d\phi) = U(d\phi)$ for every $\phi \in C_c^\infty(\R^2,M^{n\times 2})$.

Define now $\Phi_n \in C_c^\infty(\R^2, M^{n\times 2})$ as $\Phi_n = (\chi_{\Omega}\Phi)\star \rho_n$, where $\rho_n$ is a mollifier. Using the standard divergence theorem for $BV$ function we obtain
\begin{eqnarray}
T(\Phi_n) - \widetilde T(\Phi_n) &=& \partial U (\Phi_n)  =  - U(d\Phi_n) = - U_0(d\Phi_n) = -\int_{\R^2} \langle u^c,d\Phi_n\rangle \, d\mathscr{L}^2 \nonumber \\
&=& -\sum_{i=1}^n \int_{\R^2} u^c_i\div (\Phi_n)_i^\perp \, d\mathscr{L}^2
=  \sum_{i=1}^n \int_{\R^2} (\Phi_n)_i^\perp \cdot Du^c_i\,. \label{boh}
\end{eqnarray}
We observe that
\begin{equation*}
T(\Phi_n) = \int_{\Sigma}\langle \Phi_n\tau, \theta\rangle\, d\Ha^1 \rightarrow T(\Phi) \quad \mbox{as }n\rightarrow + \infty
\end{equation*}
because $\Sigma \subset \Omega$ and similarly for $\widetilde T$. Therefore taking the limit on both sides of \eqref{boh} we get
\begin{equation}
T(\Phi) - \widetilde T(\Phi) =  \sum_{i=1}^n \int_{\Omega} \Phi_i^\perp \cdot Du^c_i\,.
\end{equation}
Finally applying the divergence theorem for approximately regular vector fields (\cite{mumford}) and using that $u^c = 0$ on $\R^2 \setminus\Omega'$ we get
\begin{equation*}
T(\Phi) - \widetilde T(\Phi) = \sum_{i=1}^n \int_{\Omega} u^c_i \div \Phi_i^\perp \, d\mathscr{L}^2 = 0
\end{equation*} 
thanks to property (i) of a calibration.

Then the proof follows the same line of Proposition 3.2 in \cite{annalisaandrea}.
\end{proof}

\bibliographystyle{amsplain}
\bibliography{cali-c-p-bis}

\end{document}